\documentclass[reqno]{amsart}
\usepackage{amsmath, amsthm, amssymb, amstext}
\usepackage[foot]{amsaddr}
\usepackage[left=2cm,right=2cm,top=1.5cm,bottom=1.5cm,includeheadfoot]{geometry}
\usepackage{hyperref,xcolor}% http://ctan.org/pkg/{hyperref,xcolor}
\hypersetup{
	pdfborder={0 0 0},
	colorlinks,
}
\usepackage{todonotes}
\DeclareUnicodeCharacter{211D}{$\mathbb{R}$}
\usepackage[backend=biber,style=numeric]{biblatex}
\usepackage{enumitem}
\usepackage{comment}
\newtheorem{theorem}{Theorem}
\newtheorem{remark}[theorem]{Remark}
\newtheorem{lemma}[theorem]{Lemma}
\newtheorem{proposition}[theorem]{Proposition}
\newtheorem{corollary}[theorem]{Corollary}
\newtheorem{definition}[theorem]{Definition}

\newcommand{\be} {\begin{equation}}
\newcommand{\ee} {\end{equation}}
\newcommand{\bea} {\begin{eqnarray}}
\newcommand{\eea} {\end{eqnarray}}
\newcommand{\Bea} {\begin{eqnarray*}}
\newcommand{\Eea} {\end{eqnarray*}}

\newcommand{\al} {\alpha}
\newcommand{\ba} {\beta}

\newcommand{\De} {\Delta}
\newcommand{\la} {\lambda}

\newcommand{\e} {\epsilon}
\newcommand{\M}{m_{\alpha}(d_1,d_2)}
\newcommand{\R}{{\mathbb R}}

\newcommand{\dx}{\,{\rm d}x}
\newcommand{\dy}{\,{\rm d}y}

\newcommand{\F}{{\mathcal F}}
\newcommand{\ra} {\rightarrow}

\newcommand{\2} {{2^*_{\mu,s}}}
\newcommand{\tw}{2_{\mu,*}}

\newcommand{\Aa}{\mathcal{A}}
\newcommand{\s}{\frac{s}{2}}
\newcommand{\tri}{(-\Delta)}
\newcommand{\h}{\mathcal{H}}
\newcommand{\Pa}{ P_{\alpha}}
\newcommand{\dpq}{\delta_p +\delta_q}
\newcommand{\dg}{\tilde{d_1}}
\newcommand{\df}{\tilde{d_2}}
\newcommand{\teo}{t_{\epsilon,0}}
\newcommand{\we}{w_{\epsilon}}
\newcommand{\shl}{S_{HL}}
\newcommand{\f}{\Phi_{\alpha}}
\newcommand{\J}{\mathcal{J}_{\alpha}}
\newcommand{\cq}{\int_{\R^N}(I_{\mu}*|u|^{2^*_{\mu,s}})|v|^ {2^*_{\mu,s}} }
\newcommand{\cqt}{\int_{\R^N}(I_{\mu}*|u|^{3})|v|^ {3} }
\newcommand{\cpq}{\int_{\R^N}(I_{\mu}*|u|^{p})|v|^{q} }
\newcommand{\D}{(d_1,d_2)}
\newcommand{\cnq}{\int_{\R^N}(I_{\mu}*|u_n|^{p})|v_n|^{q} }
\newcommand{\cqn}{\int_{\R^N}(I_{\mu}*|u_n|^{3})|v_n|^{3} }

\newcommand{\tcqn}{\int_{\R^N}(I_{\mu}*|\tilde{u}_n|^{3})|\tilde{v}_n|^{3} }
\newcommand{\uvn}{(u_n,v_n)}

\newcommand{\twe}{t_{\epsilon,\alpha}}

\numberwithin{theorem}{section} \numberwithin{equation}{section}

\title{Normalized solution to Kirchhoff-fractional system involving  critical Choquard nonlinearity}
\author{divya goel, Shilpa Gupta, Asmita Rai}
\address{Divya Goel \newline
	Department of Mathematical Sciences, Indian Institute of Technology (BHU), Varanasi, 221005, India.}  
\email{divya.mat@iitbhu.ac.in}

\address{Shilpa Gupta  \newline
	Department of Mathematics and Statistics, Indian Institute of Technology Kanpur, Kanpur, 208016, India}
\email{shilpagupta890@gmail.com}
\address{Asmita Rai \newline
	Department of Mathematical Sciences, Indian Institute of Technology (BHU), Varanasi, 221005, India.}  
\email{asmita.rai65@gmail.com}
\begin{document}

\begin{abstract}
In this article, we explore  the fractional Kirchhoff-Choquard system given by
\begin{equation*}
\begin{cases}
(a+b\int_{\mathbb{R}^N}|\tri^{\frac{s}{2}} u|^2\dx)\tri^su=\lambda_1u+(I_{\mu}*|v|^{\2})|u|^{\2-2}u +\alpha p (I_{\mu}*|v|^{q})|u|^{p-2}u  ~~\text{in}~\mathbb{R}^N,\\
(a+b\int_{\mathbb{R}^N}|\tri^{\frac{s}{2}} v|^2\dx)\tri^sv=\lambda_2v+
(I_{\mu}*|u|^{\2})|v|^{\2-2}u +\alpha q(I_{\mu}*|u|^{p})|v|^{q-2}v ~~\text{in}~\mathbb{R}^N,\\
\int_{\mathbb{R}^N}|u|^2=d_1^2,~~\int_{\R^N}|v|^2=d_2^2.
\end{cases}
\end{equation*}
where  $ N> 2s,~ s \in (0,1),~ \mu \in (0, N),~\alpha \in\R$. Here, $ I_{\mu}:\R^N \to \R$ denotes the Riesz potential.  We denote by $2_{\mu,*}:=\frac{2N-\mu}{N}$ and $\frac{2N-\mu}{N-2s}:=\2$, the lower and upper Hardy–Littlewood–Sobolev critical exponents, respectively, and assume that $2_{\mu,*} < p, q< \2$. Our primary focus is on the existence of normalized solutions for the case $\alpha>0$ in two scenarios: the $ L^2-$subcritical case characterized by $2\tw<p + q < 4 + \frac{4s-2\mu}{N}$  and $L^2$-supercritical associated with $4+\frac{8s-2\mu}{N}< p + q < 2\2$.\\

\noindent \textbf{Key words:} Kirchhoff-type system, Fractional coupled Choquard system, Normalized ground states, critical exponents, Variational methods, Pohozaev manifold
\medskip

\noindent \textbf{2020 Mathematics Subject Classification:} 35B33, 35J50, 35J60, 35R11, 45K05
\end{abstract}
%\tableofcontents
\maketitle
\section{Introduction}
The motivation of the problem studied in the paper is to find a stationary wave solution to the following model
\begin{equation*}   \begin{cases}
\displaystyle \left(M(\|\tri^{\s}\phi_1\|^2_2) \right) (-\Delta)^s \phi_1 = \iota \frac{\partial \phi_1}{\partial t} +(I_{\mu}*|\phi_2|^{2^*_{\mu,s}})|\phi_1|^{2^*_{\mu,s}-2}\phi_1 +\alpha p (I_{\mu}*|\phi_2|^{q})|\phi_1|^{p-2}\phi_1 \; \text{in} \; \R^N,\\
 \left(M(\|\tri^{\s}\phi_2\|^2_2) \right) (-\Delta)^s \phi_2 = \iota \displaystyle\frac{\partial \phi_2}{\partial t} +(I_{\mu}*|\phi_1|^{2^*_{\mu,s}})|\phi_2|^{2^*_{\mu,s}-2}\phi_2 +\alpha q (I_{\mu}*|\phi_1|^{p})|\phi_2|^{q-2}\phi_2 \; \text{in} \; \R^N,\\
 \phi_1 (x,t),~\phi_2 (x,t) \ra 0 \text{ as } |x|\ra \infty,
\end{cases}
\end{equation*}
where $M(t):=a+bt$ is known as Kirchhoff function. The function $\phi_1,~\phi_2$ are the
wave function components, and the mass of each represents the
number of particles of each component in the mean-field models of Bose-Einstein condensation, see \cite{bagnato2015bose,frantzeskakis2010dark} and references therein. The real numbers $\la_1,~\la_2$ and $\al$ denote
the intraspecies and interspecies scattering lengths, describing respectively the interaction
between particles of the same component or of different components. Precisely, the
positive sign of  $\la_i$ and $\al$ is for attractive interaction, while the negative sign is for repulsive interaction. The standing wave solution of the above problem is identified with the following ansatz
\begin{equation*}
    \phi_1= e^{\iota \la_1 t} u, \quad 
    \phi_2 = e^{\iota \la_2 t} v  \text{ for } \la_1, \la_2 \in \R.
\end{equation*}
This ansatz  leads to the following
elliptic system for the densities  $u$ and $v$
\begin{equation*}
\begin{cases}
(a+b\int_{\mathbb{R}^N}|\tri^{\frac{s}{2}} u|^2\dx)\tri^su=\lambda_1u+(I_{\mu}*|v|^{\2})|u|^{\2-2}u +\alpha p (I_{\mu}*|v|^{q})|u|^{p-2}u  ~~\text{in}~\mathbb{R}^N,\\
(a+b\int_{\mathbb{R}^N}|\tri^{\frac{s}{2}} v|^2\dx)\tri^sv=\lambda_2v+
(I_{\mu}*|u|^{\2})|v|^{\2-2}u +\alpha q(I_{\mu}*|u|^{p})|v|^{q-2}u ~~\text{in}~\mathbb{R}^N,\\
\int_{\mathbb{R}^N}|u|^2=d_1^2,~~\int_{\R^N}|v|^2=d_2^2,
\end{cases}
\end{equation*}
where $ N> 2s,~ s \in (0,1),~ \mu \in (0, N),~\alpha \in\R$, $\lambda_1,~\lambda_2$ appears as Lagrange multiplier, $2_{\mu,*} :=\frac{2N-\mu}{N}  < p, ~q <\frac{2N-\mu}{N-2s}=:\2$, $I_{\mu}:=|x|^{-\mu}$ is a Riesz potential of order $\mu \in (0,N)$ %defined as 
%\begin{equation*}
%I_{\mu}(x)=\dfrac{\Gamma(\frac{\mu}{2})}{\Gamma(\frac{N-\mu}{2})\pi^{\frac{N}{2}}2^{N-\mu}|x|^{\mu}},~~for ~each~x\in \R\backslash\{0\} ,
%\end{equation*}
and the operator $\tri^s$ is fractional Laplacian defined by 
\begin{equation*}
\tri^su(x)=P.V.\int_{\R^N}\frac{u(x)-u(y)}{|x-y|^{N+2s}}\dy
\end{equation*}
where $ P.V. $ represents the Cauchy principal value of the integral. For more details on the operator, one can refer \cite{di2012hitchhikers}. 

The study of doubly nonlocal systems with fixed frequencies $\lambda_{1}$ and $\lambda_{2}$ has been extensively developed over the past two decades. For instance, Guo, Luo, and Zou \cite{guo2017critical} investigated the existence of ground state solutions for a critical system involving the fractional Laplacian with prescribed frequencies $\lambda_{1}$ and $\lambda_{2}$. Related contributions can be found in several other works \cite{ambrosio2020multiplicity,fiscella2018p}, where various aspects of such systems have been explored under different settings. On the other hand, when the $L^{2}$-norm of the unknown functions $u$ and $v$ is fixed, the parameters $\lambda_{1}$ and $\lambda_{2}$ arise naturally as Lagrange multipliers. In this framework, any solution obtained with prescribed mass is referred to as a \emph{normalized solution}. From a physical perspective, the prescribed mass condition embodies the principle of conservation of mass, which underlines the particular significance of normalized solutions in the study of these systems.
There is a lot of work on Normalized solution from last decade.  
We introduce some results about the existence of normalized solutions to the
semilinear Schrödinger equation
\begin{equation*} 
(-\Delta)^s u = \lambda u + h(u) \quad \text{in } \mathbb{R}^N. 
\end{equation*}
 In the local case ($s=1$),  $ h(u) = |u|^{p-2}u$, the associated energy functional takes the form:
\begin{equation*}
E(u) = \frac{1}{2} \int_{\mathbb{R}^N} |\nabla u|^2 \, dx - \frac{1}{p} \int_{\mathbb{R}^N} |u|^p \, dx.
\end{equation*}
For $2+\frac{4}{N}<p<2^* :=\frac{2N}{N-2}$, Jeanjean~\cite{jeanjean1997existence} constructed a variational framework involving a mountain pass geometry on a constraint manifold, and through this structure, combined with compactness arguments, proved the existence of normalized solutions. Later, Jeanjean and Lu~\cite{jeanjean2020mass} extended these results to show infinitely many radial solutions using a minimax principle. 

In the recent past,  Soave~\cite{soave2020normalized1} studied normalized solutions in the presence of Sobolev critical growth and mixed nonlinearities of the form  
$h(u) = \alpha |u|^{q-2}u + |u|^{p-2}u,$ 
with $ \alpha \in \mathbb{R},~ 2 < q \le 2 + \frac{4}{N} \le p < 2^* $. His work comprehensively classified the problem in terms of $ L^2$-subcritical, critical, and supercritical regimes. He interplays with $L^2$ subcritical and $L^2$ supercritical terms and proved the multiplicity of solutions. Subsequently, in a follow-up paper~\cite{soave2020normalized2}, Soave also addressed the critical case. Due to the lack of compactness in the Sobolev embedding $ H^1(\mathbb{R}^N) \hookrightarrow L^{2^*}(\mathbb{R}^N) $, this problem is more challenging. He navigated these difficulties by drawing from techniques used in~\cite{brezis1983positive,jeanjean1997existence}. He obtained a constrained Palais–Smale sequence with an additional property by studying the geometry of the corresponding Pohozaev manifold, and he proved the compactness of this special constrained Palais–Smale sequence under some energy level.

Motivated by these developments, several researchers have extended the theory to incorporate the fractional Laplacian, thereby broadening the scope of the normalized solution framework. The initial results on the existence of such solutions under the condition that the associated energy functional is bounded from below were established in \cite{bhattarai2017fractional,feng2019existence}. A significant advancement was made by Yang \cite{yang2020normalized}, who, for the first time, proved the existence of normalized solutions even when the energy functional is unbounded from below. Continuing along this direction, Zhen and Zhang \cite{zhen2022normalized} employed Jeanjean’s variational technique to obtain existence results for normalized solutions in the presence of mixed nonlinearities.
The work on the nonlinear Schrödinger system, focused on the existence of normalized ground solutions under mass constraints with coupled nonlinearity, is gaining attention among researchers these days. In \cite{bartsch2016normalized}, Bartsch, Jeanjean, and Soave studied the following system 
\begin{equation}\label{a}
    \begin{cases}
        -\Delta u + \lambda_1 u = \mu_1 |u|^{l_1 - 2} u + \nu \alpha |u|^{\alpha - 2} |v|^{\beta} u, & \text{in } \mathbb{R}^N, \\[6pt]
        -\Delta v + \lambda_2 v = \mu_2 |v|^{l_2 - 2} v + \nu \beta |v|^{\beta - 2} |u|^{\alpha} v, & \text{in } \mathbb{R}^N, \\[6pt]
        \displaystyle\int_{\mathbb{R}^N} |u|^2 \, dx = d_1^2,\quad 
        \displaystyle\int_{\mathbb{R}^N} |v|^2 \, dx = d_2^2.
    \end{cases}
\end{equation}
\begin{comment}

\begin{equation}\label{a}
    \begin{cases}
        -\Delta u+\lambda_1u=&\mu_1|u|^{l_1-2}u+\nu\alpha|u|^{\alpha-2}|v|^{\beta}u~\text{in}~\R^N,\\
        -\Delta v+\lambda_2v=&\mu_2|v|^{l_2-2}v+\nu\beta|v|^{\beta-2}|u|^{\alpha}v~\text{in}~\R^N,\\
        \int_{\R^N}|u|^2\dx=d_1^2
 ,&\int_{\R^N}|v|^2\dx=d_2^2.
 \end{cases}
\end{equation}
\end{comment}
with  $l_1=4,~l_2=4$, $\al=2,~\beta=2$ and $N=3$ proved the existence
of positive normalized solutions for different ranges of the coupling parameter $\nu > 0$, without any assumption on the masses. Later \cite{bartsch2019multiple}, Bartsch and Soave proved the existence of infinitely many normalised solutions of \eqref{a} 
 by using the Krasnoselskii genus approach for $d_1=d_2$ and $\mu_1=\mu_2$. For the Sobolev critical nonlinearity, that is, $l_1=l_2=2^*$, and $\alpha+\beta<2^*$,  Bartsch-Li-Zou \cite{bartsch2023existence} proved the existence and asymptotic properties of normalized ground states of \eqref{a}. 
 Zhang and Han \cite{zhang2024normalized} deal with the system in the Sobolev critical case of coupled nonlinearities, i.e., when $\alpha+\beta=2^*$. However, their main focus was on the $L^2$-subcritical case, i.e., $2\le l_1,~l_2<2+\frac{4}{N}$. For more details on normalized solutions for systems, one can refer to \cite{MENG2025113845}.

\begin{comment}
    
Several researchers attempted to extend the conclusions of a single equation to a system of equations. Consider the following system of equations 
\begin{equation}\label{a}
    \begin{cases}
        -\Delta u+\lambda_1u=&|u|^{l_1-2}u+\nu\alpha|u|^{\alpha-2}|v|^{\beta}u~\text{in}~\R^N,\\
        -\Delta v+\lambda_2v=&|v|^{l_2-2}v+\nu\beta|v|^{\beta-2}|u|^{\alpha}v~\text{in}~\R^N,\\
        \int_{\R^N}|u|^2\dx=d_1^2
 ,&\int_{\R^N}|v|^2\dx=d_2^2.
 \end{cases}
\end{equation}
For the Sobolev subcritical case of \eqref{a}, one can refer \cite{bartsch2016normalized,bartsch2021normalized,bartsch2019multiple,gou2018multiple} for the existence of a normalized solution. Bartsch, Li and Zou \cite{bartsch2023existence} discuss the Sobolev critical case i.e. for $l_1=l_2=2^*$ assuming $N\in\{3,4\}$,$\alpha,\beta>1$ and $\alpha+\beta<2^*$, they proved existence  for $\nu\in(0,\nu_0)$, for some $\nu_0>0$ and non-existence of solution when $\nu<0$. Zhang and Han \cite{zhang2024normalized} deal with the system when coupled nonlinearities Sobolev critical case, i.e., when $\alpha+\beta=2^*$, but they mainly focus on the $L^2$-subcritical case, i.e., $2\le l_1,l_2<2+\frac{4}{N}$. 
\end{comment}

Of course, a great deal of work has focused on the normalized solution of  Schr\"{o}dinger systems involving the fractional Laplacian.  In particular, we highlight that  Lui and Li \cite{liu2023mass}  obtained the
existence of normalized solutions  for the  following system    under the assumption   when $N>2,~\frac{1}{2}\le s<1,$ and $2+4s/N<p,~q,~r_1+r_2<2^*_s:=2N/(N-2s)$ 
   \begin{equation}\label{as}
\begin{cases}
(-\Delta)^s u = \lambda_1 u + \mu_1 |u|^{p-2} u + \beta r_1 |u|^{r_1-2} u |v|^{r_2}, & \text{in } \mathbb{R}^N, \\[6pt]
(-\Delta)^s v = \lambda_2 v + \mu_2 |v|^{q-2} v + \beta r_2 |u|^{r_1} |v|^{r_2-2} v, & \text{in } \mathbb{R}^N,\\[4pt]
\displaystyle\int_{\mathbb{R}^N} |u|^2 \dx = d_1^2, \quad \int_{\mathbb{R}^N} |v|^2 \dx = d_2^2. 
\end{cases}
\end{equation}
 Authors also established the non-existence in the case when $p=q=rd_1+r_2=2^*_s.$

Recently, Chen and Yang \cite{chennormalized} studied the following fractional Choquard system with the Sobolev critical exponent
   \begin{equation}\label{as1}
\begin{cases}
(-\Delta)^s u + \lambda_1 u = (I_\mu * |u|^p)|u|^{p-2}u + \frac{\beta r_1}{2 s} |u|^{r_1-2}u |v|^{r_2}, & \text{in } \mathbb{R}^N, \\[4pt]
(-\Delta)^s v + \lambda_2 v = (I_\mu* |v|^q)|v|^{q-2}v + \frac{\beta r_2}{2 s} |u|^{r_1} |v|^{r_2-2} v, & \text{in } \mathbb{R}^N, \\
\displaystyle\int_{\mathbb{R}^N} |u|^2 \dx = d_1^2, \quad \int_{\mathbb{R}^N} |v|^2 \dx = d_2^2,
\end{cases}
\end{equation}
  $2_{\mu,*}<p,~q<\2$ and $r_1+r_2=2_s^*$. They focus on ground state and mountain pass type solutions for the Sobolev critical coupling term. For the subcritical Choquard growth, we refer the readers to \cite{chennormalized,guo2024normalized2,guo2024normalized,wang2021existence} and the references therein. However, for the critical coupled nonlinearity, so far we have found only one paper \cite{zhang2023normalized}  dealing with double-coupled Choquard nonlinearity. Precisely, Zhang, Zhang, and Zhong \cite{zhang2023normalized}  extended the results of  \cite{chennormalized} for the coupled critical Choquard nonlinearity for $s=1$ and obtained the ground state solutions for $p+q\leq 4+\frac{4-2\mu}{N}$. While for the $L^2$ super-critical case, the authors prove the existence of a mountain pass type solution. 
\\
 In \cite{yang2021normalized}, Yang studied the following  Kirchhoff system, 
\begin{equation}
\begin{cases}
-(a_1+b_1\int_{\R^N}|\nabla u|^2)\Delta u+\lambda_1u=\mu_1|u|^{p-2}u+r_1\beta|u|^{r_1-2}|v|^{r_2}u&\text{in}~\R^N,\\[4pt]
-(a_2+b_2\int_{\R^N}|\nabla v|^2)\Delta v+\lambda_2v=\mu_2|v|^{q-2}v+r_2\beta|v|^{r_2-2}|u|^{r_1}v&\text{in}~\R^N,\\
\displaystyle\int_{\R^N}|u|^2=d_1^2~~\int_{\R^N}|v|^2=d_2^2.
\end{cases}
\end{equation}
He derives the existence of a solution for the case when $2+\frac{8}{N}<r_1+r_2<p,~q<2^*$. For the case where $r_1+r_2=2^*$, Zhang and Zhang \cite{zhang2025normalized} prove the existence of a solution, but only for the case when $2\le p,~q<2+\frac{8}{N}$ for any $N>5$. For more work one can refer \cite{cao2017existence,KongChen2022}. 

   To the best of our knowledge, all the works which have been done on the Kirchhoff system are mostly of  $s=1$, except for a few. In \cite{li2024existence}, researchers extended the results of 
\cite{yang2021normalized} for the fractional case, but prove the existence of normalized solutions only when nonlinearity has subcritical Sobolev growth. In this article, we proved the existence of ground state solutions for any $N>2s$ with perturbation of $L^2$-subcritical growth.  While for $L^2$- supercritical growth we prove the existence of mountain pass theorem. These types of results for the Kirchhoff operator are proved in \cite{li2022normalized} for a single equation. However, for systems of equations, this question has remained open. In this paper, we answer this problem positively.  Taking motivation from the above papers here we took the coupled perturbation and coupled critical nonlinearity in the sense of Hardy-Littlewood-Sobolev exponent. Precisely, we studied the following problem 
   \begin{equation*}    
   \begin{cases}
\displaystyle \left(M(\|\tri^{\s}u\|^2_2) \right) (-\Delta)^s u = \lambda_1u +(I_{\mu}*|v|^{2^*_{\mu,s}})|u|^{2^*_{\mu,s}-2}u +\alpha p (I_{\mu}*|v|^{q})|u|^{p-2}u \; \text{in} \; \R^N,\\[4pt]
%\quad \quad u > 0\quad\text{in} \; \Omega,\\
 \left(M(\|\tri^{\s}v\|^2_2) \right) (-\Delta)^s v = \lambda_2v +(I_{\mu}*|u|^{2^*_{\mu,s}})|v|^{2^*_{\mu,s}-2}v +\alpha q (I_{\mu}*|u|^{p})|v|^{q-2}v \; \text{in} \; \R^N,\\[4pt]
\displaystyle\int_{\mathbb{R}^N}|u|^2=d_1^2,~~\int_{\R^N}|v|^2=d_2^2,
\tag{\(F_d\)}
\label{G}
  \end{cases}
\end{equation*}
where  $N>2s,~s \in (0,1),~\mu \in (0,N), ~\al>0,~$ $2_{\mu,*} :=\frac{2N-\mu}{N}  < p, ~q <\frac{2N-\mu}{N-2s}=:\2$ and $\lambda_1,~\lambda_2 \in \R$ and $M(t):=a+bt,~a,~b >0$. 
This article is the first to prove the existence of normalized solutions in the double critical regime (involving both Kirchhoff's nonlocal term and the HLS critical exponent for the nonlinearity). Here, we used the unified techniques to prove the boundedness and strong convergence of the Palais-Smale sequences.  Further to achieve the goal, we did the 
analysis of the fine energy estimates, which includes the minimizers of the best constant $S_{HL}$  (See Lemma \ref{lemma_5.6}). Subsequently, we used the minimization technique and the mountain pass theorem to prove the existence result. We like to mention that the results we prove in this article are new for $s=1$.

With this introduction, we would like to state our main results 
 \begin{theorem}\label{AAA}
Let $22_{\mu,*}<p+q<4+\frac{4s-2\mu}{N}$ and $0<\mu<min\{\alpha_*,\alpha^*\,\alpha_*^*\}$. Then
\begin{itemize}
	\item[(i)] $\J|_{S\D}$ admits a critical point  $(\tilde{u},\tilde{v})$ at an energy level $m_{\alpha}\D<0$. This critical point is a local minimizer of $\J$ on the set 
	$$\mathcal{B}_{R^*}:=\{(u,v)\in S\D:([u]_s^2+[v]_s^2)^{1/2}<R^* \},$$ 
for some fixed $R^*>0$. Moreover, this solution is a ground state of $\J|_{S\D}$, and any ground
state of $\J|_{S\D}$ is a local minimizer of $\J$ on $\mathcal{B}_{R^*}$;
	\item[(ii)] The critical point $(\tilde{u},\tilde{v})$ solves \eqref{G} for some $\tilde{\lambda}_1,~\tilde{\lambda}_2<0$, and it is positive and radially decreasing.
    \item[(iii)] If $(\tilde{u},\tilde{v})\in S\D$ is a ground state for $\J|_{S\D}$, then as $\alpha\to 0^+$ we have $m_{\alpha}\D\to0,~([\tilde{u}]_s^2+[\tilde{v}]_s^2)\to0$.
\end{itemize} 	
 \end{theorem}  
\begin{theorem}\label{AA2}
	Let $\alpha>0$, $4+\frac{8s-2\mu}{N}<p+q<2\2=6$. If one of the following conditions holds:
    \begin{itemize}
    \item[(i)]$p>\frac{3}{2},~q>\frac{3}{2}$ and either $4s<N<6s$,   $\alpha<\tilde{\alpha}$ or $\frac{22s-4(p+q)s}{6-(p+q)}< N\le 4s,$
    \item[(ii)] $p>\frac{3}{2},~q<\frac{3}{2}$ and $\frac{16s-4ps}{6-(p+q)}<N<\frac{18s-4ps}{6-(p+q)}$,
    \item[(iii)] $p<\frac{3}{2},~q>\frac{3}{2}$ and $\frac{16s-4qs}{6-(p+q)}<N<\frac{18s-4qs}{6-(p+q)}$,
    \item[(iv)] $p>\frac{3}{2}$, $q=\frac{3}{2}$ and $\frac{8s(4-p)}{9-2p}<N\le4s$,
    \item[(v)]  $p=\frac{3}{2}$, $q>\frac{3}{2}$ and $\frac{8s(4-q)}{9-2q}<N\le4s$. 
    \end{itemize}
    Then, the functional $\mathcal{J}_{\alpha}$ restricted to the set $S(d_1,d_2)$ possesses a critical point of mountain pass type at a level $\sigma_{\alpha}$ satisfying
	$$\sigma_{\alpha}\D\in \left(0,\frac{ab\shl^3}{2}+\frac{b^3\shl^6}{12}+\frac{2}{3}\left(\frac{b^2\shl^4}{4}+a\shl \right)^{\frac{3}{2}} \right).$$ This critical point corresponds to a positive radial solution of the equation \eqref{G}.
\end{theorem}
\begin{remark}
Using the  techniques and ideas of this article, one can easily prove the Theorems \ref{AAA} and \ref{AA2} for the  following problem 
     \begin{equation*}    
   \begin{cases}
\displaystyle \left(M(\|\tri^{\s}u\|^2_2) \right) (-\Delta)^s u = \lambda_1u + |v|^{r_2}|u|^{r_1-2}u +\alpha p |v|^{q}|u|^{p-2}u \; \text{in} \; \R^N,\\[4pt]
%\quad \quad u > 0\quad\text{in} \; \Omega,\\
 \left(M(\|\tri^{\s}v\|^2_2) \right) (-\Delta)^s v = \lambda_2v +|u|^{r_1}|v|^{r_2-2}v +\alpha q |u|^{p}|v|^{q-2}v \; \text{in} \; \R^N,\\[4pt]
\displaystyle\int_{\mathbb{R}^N}|u|^2=d_1^2,~~\int_{\R^N}|v|^2=d_2^2,
  \end{cases}
\end{equation*}
where  $N>2s,~s \in (0,1), ~\al>0,~$ $2 < p,~ q <\frac{2N}{N-2s}=:2^*_s,~r_1+r_2= 2^*_s$ and $\lambda_1,~\lambda_2 \in \R$ and $M(t):=a+bt,~a,~b >0$. 
\end{remark}

The following are the notations that are used in the article
 \begin{align*}
  \left(\int_{\R^N}|u|^p\dx\right)^{\frac{1}{p}}:=\|u\|_p,  
 \end{align*}
 \begin{align*}
     \tilde{\alpha}:=\frac{b}{2^{p+q}(\dpq) B_{N,p,q,s,\mu}d_1^{p+q-(\dpq)}\big((b\shl^3)^{\frac{\dpq-4}{2}}+(a\shl^3)^{\frac{\dpq-4}{4}}\big)}
 \end{align*}
 	\begin{align*}
 \alpha_*:=\frac{1}{B_1} \left(\omega_1^{{\frac{4-(\dpq)}{2\2-4}}}- \omega_1^{\frac{2\2-(\dpq)}{2\2-4}}   \right)\left(\frac{a}{2}\left( \frac{b(2\shl)^{\2}}{8\2}  \right)^{\frac{2-(\dpq)}{2\2-4}}+\frac{(\frac{b}{8})^{\frac{2\2-(\dpq)}{2\2-4 }} }{\left(2(2\shl ) ^{\2}\right)^{{\frac{\dpq-4}{ 2\2-4} } }}  \right),
 \end{align*}
 \begin{equation}\label{al_2}
	\alpha^*:=\left(\dfrac{b(4-(\dpq))}{4(2\2 -(\dpq))}(2\shl)^{\2 }  \right)^{\frac{4-(\dpq)}{2(\2-2)}}\left( \dfrac{b}{(\dpq)(2\2 -(\dpq))B_1}\right),
	\end{equation}
    and
    \begin{equation*}
        \alpha_*^*:=\frac{\left(\frac{b}{\dpq}\right)^{\frac{\dpq}{4}}}{B_1\big(6-(\dpq))}\left(\frac{24}{4-(\dpq)}\left(\frac{ab\shl^3}{2}+\frac{b^3\shl^6}{12}+\frac{2}{3}\left(\frac{b^2\shl^4}{4}+a\shl \right)^{\frac{3}{2}}\right)\right)^{1-\frac{\dpq}{4}},
    \end{equation*}
 where, $\omega_1:=\left(\frac{2(4-(\dpq))}{\2(\2-1)(2\2-(\dpq))}\right)$ and $B_1=B_{N,p,q,s,\mu}(d_1^2+d_2^2)^{\frac{p+q-(\dpq)}{2}}$.

\newpage
 
\textbf{Structure of the Paper.} The layout of the paper is organized as follows:
\begin{itemize}
    \item In Section~\ref{S1}, we revisit some classical foundational results, and introduce the variational setting related to the problem \eqref{G}.
    \item In Section~\ref{S3}, we examine the Palais-Smale condition and provide a detailed compactness analysis.
    \item Section~\ref{S4} contains the complete proof of Theorem~\ref{AAA}.
    \item Finally, Section~\ref{S5} is devoted to establishing the proof of Theorem~\ref{AA2}.
\end{itemize}
\maketitle
\section{Preliminaries and Variational Framework}\label{S1}
This section discusses key results and the variational framework that will serve as fundamental tools in the subsequent sections.\\
The space $H^s(\R^N)$ represents the usual fractional Sobolev space, which is a Hilbert space equipped with the standard inner product defined as
$$\langle u,v \rangle_{H^s(\R^N)}=\Aa(u,v)+\int_{\R^N}u v\dx$$
and norm assigned to it \cite{di2012hitchhikers} as follows 
\begin{equation*}
\|u\|_{H^s(\R^N)}=\Big([u]_s^2+\int_{\R^N}|u|^2\dx \Big)^{\frac{1}{2}},
\end{equation*}
where
$$\Aa(u,v):=\int_{\R^N} \int_{\R^N} \frac{(u(x)-u(y))(v(x)-v(y))}{|x-y|^{N+2s}}\dx\dy$$
and
$$[u]_s^2:=\mathcal{A}(u,u)=\int_{\R^N}\int_{\R^N}\frac{|u(x)-u(y)|^2}{|x-y|^{N+2s}}\dx\dy.$$

Consider the  product space $\mathcal{H}:=H^s(\R^N)\times H^s(\R^N)$, which is a normed space  with the norm 
\begin{equation*}
\|(u,v)\|_{\h}=(\|u\|_{H^s(\R^N)}+\|v\|_{H^s(\R^N)})^{\frac{1}{2}}.
\end{equation*}
\begin{proposition}\label{AK2}(Hardy-Littlewood-Sobolev inequality)\cite{lieb2001analysis}
	 Let $p,q>1$ and $0<\mu < N$ with $ \frac{1}{p}+\frac{\mu}{N}+\frac{1}{q}=2,$ $ f \in L^s(\mathbb{R}^N) $ and $ h\in L^r(\mathbb{R}^N) $. Then there exists a sharp constant $C(N,\mu)$, independent of $f$ and $h$, such that 
	\begin{equation*}
	\int_{\mathbb{R}^N} \int_{\mathbb{R}^N} \dfrac{f(x)h(y)}{|x-y|^{\mu}}\dy\dx \leq C(N,\mu)\|f\|_{L^p(\mathbb{R}^N)}\|h\|_{L^q(\mathbb{R}^N)}.
	\end{equation*}
	%If $p=q=\frac{2N}{2N-\mu}$, %then
	%\begin{equation*}
	%C(p,N,\mu , q)=\pi^{\frac{\mu}{2}}\dfrac{\Gamma (\frac{N-\mu}{2})}{\Gamma(N-\frac{\mu}{2})}\Big(\frac{\Gamma(\frac{N}{2})}{\Gamma(N)}\Big)^{-1+\frac{\mu}{N}}.
	%\end{equation*}
\end{proposition}
\begin{proposition}\cite{mukherjee2016fractional}
	There exists a constant $\shl>0 $ such that
	\begin{equation}\label{Ai}
	\shl=\inf_{u\in H^s(\R^N)\backslash\{0\}}\frac{[u]_s^2}{(\int_{\R^N}(I_{\mu}*|u|^{\2})|u|^{\2})^{\frac{1}{\2}}}.
	\end{equation}
\end{proposition}
Using (3.3) of \cite{ghimenti2016nodal} and \eqref{Ai},  we have 
\begin{align}\label{Ai2}
\notag \cq &\le \left(\int_{\R^N}(I_{\mu}*|u|^{\2})|u|^{\2}\right)^{\frac{1}{2}}\left(\int_{\R^N}(I_{\mu}*|v|^{\2})|v|^{\2}\right)^{\frac{1}{2}}\\
\notag
&\le \left( \frac{[u]_s^{\2}[v]_s^{\2}}{\shl^{\2}}\right)\\
&\le \frac{1}{\shl^{\2}}\left( \frac{[u]_s^2+[v]_s^2}{2}\right)^{\2}.
\end{align}
\begin{proposition}
\cite[Lemma 2.4]{feng2018stability} Let $\tw <p<\2$. Then 
\begin{equation}\label{Ai6}
\int_{\R^N}(I_{\mu}*|u|^p)|u|^p\le 
B_{N,p,s,\mu}[u]_s^{2\delta_p}\|u\|_2^{2p-2\delta_p}.
\end{equation}
where $$\delta_p :=\frac{N(p-2)+\mu}{2s} \text{  and  } \delta_q :=\frac{N(q-2)+\mu}{2s},$$    for some $ B_{N,p,s,\mu}>0$.
\end{proposition}
Then, again  using \cite{ghimenti2016nodal} and \eqref{Ai6}, we obtain	
	\begin{equation}\label{Ai4}
	\cpq \le B_{N,p,q,s,\mu}\left( [u]_s^2+[v]_s^2\right)^{\frac{\dpq}{2}}\left(\|u\|_2^2+\|v\|_2^2\right)^{\frac{p+q-(\dpq)}{2}}.
	\end{equation}
We say that $(u,v) \in \h$ is a weak solution of \eqref{G} if 
\begin{align*}
\left(a+b[u]_s^2\right)\Aa(u,\phi)+\left(a+b[v]_s^2\right)\Aa(v,\xi)-\lambda_1\int_{\R^N}u\phi \dx-\lambda_2 \int_{\R^N}v\xi \dx
-\int_{\R^N}(I_{\mu}*|v|^{\2})|u|^{\2-2}u\phi\\ -\int_{\R^N}(I_{\mu}*|u|^{\2})|v|^{\2-2}v\xi
-\alpha p\int_{\R^N}(I_{\mu}*|v|^{q})|u|^{p-2}u\phi-\alpha q \int_{\R^N}(I_{\mu}*|u|^{p})|v|^{q-2}v\xi=0,
\end{align*} 
for all $(\phi,\xi) \in \h$.\par 
The energy functional $\J :\h \to \R$ associated with the problem \eqref{G} as
\begin{align*}
\J(u,v)=\frac{a}{2}\left([u]_s^2+[v]_s^2\right)+\frac{b}{4}\left([u]_s^4+[v]_s^4\right)
-\frac{1}{\2}\cq-\alpha\cpq,
\end{align*}
on the constraint set
$$S(d_1,d_2):=\{(u,v)\in \h:\|u\|_2^2=d_1^2,\|v\|_2^2=d_2^2 \}.$$
Define $S_r(d_1,d_2):=S\D\cap\h_r$, where 
\begin{equation*}
    \h_r:=\{(u(x),v(x))\in\h:~u(x)=u(|x|), \ v(x)=v(|x|)\}.
\end{equation*}
Furthermore, solutions to \eqref{G} satisfy the Pohozaev identity
\begin{align}\label{A1}
\Pa(u,v):=a\left([u]_s^2 + [v]_s^2\right) +b\left([u]_s^4+[v]_s^4\right) -2\cq
-\alpha(\delta_p +\delta_q) \cpq=0.
\end{align}
Consider the Pohozaev manifold: 
$$\mathcal{P} (d_1,d_2):=\{(u,v)\in S(d_1,d_2):\Pa(u,v)=0\},$$
and
$$\M :=\inf_{\mathcal{P}(d_1,d_2)}\J(u,v), \ m_{r,\alpha}(d_1,d_2):=\inf_{\mathcal{P}(d_1,d_2) \cap\h_r}\J(u,v).$$
A normalized ground state of \eqref{G} is a solution $(u,v)\in \mathcal{P}(d_1,d_2)$ of \eqref{G} that achieves $\M $.

The  $L^2$-invariant scaling is defined as $k\star u(x):=e^{\frac{N}{2}k}u(e^kx)$ and $k\star(u,v):=(k\star u,k\star v)$. For $\alpha \in \R$ and $(u,v)\in S(d_1,d_2)$, we consider the map $\f:\R \to \R$ given by
\begin{equation*}
\begin{aligned}
\notag\f(k):=\J(k \star (u,v))=&\frac{a}{2}e^{2sk}([u]_s^2 +[v]_s^2) +\frac{b}{4}e^{4sk}([u]_s^4 + [v]_s^4)
-\frac{1}{\2}e^{2s\2 k}\cq \\&-\alpha e^{s(\delta_p + \delta_q)k}\cpq.
\end{aligned}
\end{equation*}
\begin{corollary}\label{a0}
	Let $\alpha \in \R^+$ and $(u,v)\in S(d_1,d_2)$, then $k\star(u,v)\in \mathcal{P}(d_1,d_2)$ if and only if $k\in \R$ is a critical point for $\f$ and the map $(u,v)\mapsto k\star(u,v)$ is of class $\mathcal{C}^1$.
\end{corollary}
By a straightforward computation, we obtain
$$\mathcal{P} \D =\{(u,v) \in S\D :(\f)'(0)=0 \}.$$
We split $\mathcal{P}\D$ into three disjoint sets as follows:
%We divide $\Pa \D $ into disjoint sets as follows:
$$\mathcal{P}^{+}\D :=\{(u,v) \in \mathcal{P} \D : (\f)''(0)>0\},$$
$$\mathcal{P}^{0}\D :=\{(u,v) \in \mathcal{P} \D : (\f)''(0)=0\},$$
$$\mathcal{P}^{-}\D :=\{(u,v) \in \mathcal{P} \D : (\f)''(0)<0\}.$$
\begin{lemma}\label{lulem13}
	\cite{mukherjee2016fractional}
	The constant $S_{HL}$ is achieved if  and only if 
	\begin{align*}
	u=C\left(\frac{b}{b^2+|x-a|^2}\right)^{\frac{N-2s}{2}}
	\end{align*} 
	where $C>0$ is a fixed constant, $a\in \mathbb{R}^N$ and $b\in (0,\infty)$ are parameters. Moreover,
	\begin{align*}
	S=	S_{HL} \left(C(N,\mu)\right)^{\frac{N-2s}{2N -\mu}}.
	\end{align*}
\end{lemma}
Consider the family of minimizers $\{U_\e\}_{\e>0}$ of $S$ defined as 
\begin{align*}
U_\e = \e^{-\frac{(N-2s)}{2}} S^{\frac{(N-\mu)(2s-N)}{4s(N-\mu+2s)}}(C(N,\mu))^{\frac{2s-N}{2(N-\mu+2s)}} u^*(x/\e)
\end{align*}
where $u^*(x)= \overline{u}(x/S^{1/2s}),\; \overline{u} (x)= \frac{\tilde{u}(x)}{|\tilde{u}|_{2^*_s}}$ and $\tilde{u}(x)= \alpha(\beta^2+|x|^2)^{\frac{-(N-2s)}{2}}$ with $ \al \in \mathbb{R}\setminus\{0\}$ and  $\ba >0$ are fixed constants. Then  from Lemma \ref{lulem13}, for $\e>0, ~U_\e$ satisfies 
\begin{align*}
(-\De)^s u = (I_{\mu}* |u|^{\2})|u|^{\2-2}u \text{  in } \R^N. 
\end{align*}
Consider a radially decreasing cut-off function $\eta\in C_0^{\infty}(\R^N) $ such that $\eta\equiv1$ in $B_1(0)$ and $\eta\equiv0$ in $\R^N\backslash B_2(0)$. Define \begin{equation}\label{a102}w_{\epsilon}:=\eta(x)U_{\epsilon}(x),
\end{equation}
we have $w_{\epsilon}\in H^s(\R^N)$.
 \begin{lemma}\label{Ak3}
        Let $s\in(0,1), N>2s, \tw<p,q<\2=3$ and $\mu\in (0,N)$, then the following holds
        \begin{itemize}
            \item[(i)] \begin{equation*}
[\we]_s^2\le \shl^{\frac{3}{2}}+O(\epsilon^{N-2s}),
\end{equation*}
            \item[(ii)]\begin{align*}
\int_{\R^N}(I_{\mu}*|\we|^{\2})|\we|^{\2}\le
\shl^{\frac{3}{2}}+O(\epsilon^{N}),\\
\int_{\R^N}(I_{\mu}*|\we|^{\2})|\we|^{\2}\ge \shl^{\frac{3}{2}}-O(\epsilon^N),
\end{align*}
            \item[(iii)]\begin{equation*}
\|\we\|_2^2=\begin{cases}
\epsilon^{2s}+O(\epsilon^{N-2s})&N>4s,\\
\epsilon^{2s}|\log\epsilon|+O(\epsilon^{2s})&N=4s,\\
\epsilon^{N-2s}+O(\epsilon^{2s})&N<4s.\\
\end{cases}
\end{equation*}
\item[(iv)]
\begin{equation*}
    \int_{\R^N}(I_{\mu}*|w_{\epsilon}|^p)|w_{\epsilon}|^q\le
    \begin{cases}
        \epsilon^{\frac{(N-2s)(6-(p+q))}{2}}&p>\frac{3}{2},~q>\frac{3}{2},\\
        \epsilon^{\frac{(N-2s)(9-2p)}{4}}|\log\epsilon|^{\frac{3(N-2s)}{2N}}&p>\frac{3}{2},~q=\frac{3}{2},\\
        \epsilon^{\frac{(N-2s)(9-2q)}{4}}|\log\epsilon|^{\frac{3(N-2s)}{2N}}&p=\frac{3}{2},~q>\frac{3}{2},\\
        \epsilon^{\frac{(N-2s)(3-p+q)}{2}}&p>\frac{3}{2},~q<\frac{3}{2},\\
        \epsilon^{\frac{(N-2s)(3+p-q)}{2}}& p<\frac{3}{2},~q>\frac{3}{2}.
    \end{cases}
\end{equation*}
\item[(v)]  
\begin{equation*}
\dfrac{\displaystyle\int_{\R^N}(I_{\mu}*|w_{\epsilon}|^p)\,|w_{\epsilon}|^q}
      {\|w_{\epsilon}\|_2^{\,p+q-(\dpq)}}
\;\le\;
\begin{cases}
\textbf{Case 1: } 4s<N<6s, \\[4pt]
\quad
\begin{cases}
 \text{Constant}~~, & p>\tfrac{3}{2},\; q>\tfrac{3}{2}, \\[4pt]
\epsilon^{\frac{(N-2s)(2\min\{p,q\}-3)}{4}}, & \big(p>\tfrac{3}{2},\,q<\tfrac{3}{2}\big)\;\text{or}\;\big(p<\tfrac{3}{2},\,q>\tfrac{3}{2}\big), \\[4pt]
|\log \epsilon|^{\frac{3(N-2s)}{2N}}, & \big(p>\tfrac{3}{2},\,q=\tfrac{3}{2}\big)\;\text{or}\;\big(p=\tfrac{3}{2},\,q>\tfrac{3}{2}\big),
\end{cases} \\[12pt]

\textbf{Case 2: } N=4s, \\[4pt]
\quad
\begin{cases}
|\log \epsilon|^{\tfrac{p+q-6}{2}}, & p>\tfrac{3}{2},\; q>\tfrac{3}{2}, \\[4pt]
\epsilon^{s(2\min\{p,q\}-3)}\,|\log \epsilon|^{\tfrac{6-(p+q)}{2}}, & \big(p>\tfrac{3}{2},\,q<\tfrac{3}{2}\big)\;\text{or}\;\big(p<\tfrac{3}{2},\,q>\tfrac{3}{2}\big), \\[4pt]
|\log \epsilon|^{\tfrac{-3+\max\{p,q\}}{2}}, & \big(p>\tfrac{3}{2},\,q=\tfrac{3}{2}\big)\;\text{or}\;\big(p=\tfrac{3}{2},\,q>\tfrac{3}{2}\big),
\end{cases}\\[12pt]
\textbf{Case 3: } 2s<N<4s, \\[4pt]
\quad
\begin{cases}
\epsilon^{\tfrac{(N-2s)(4s-N)(6-(p+q))}{4s}}, & p>\tfrac{3}{2},\; q>\tfrac{3}{2}, \\[4pt]
\epsilon^{\tfrac{(N-2s)}{4s}\big(18s+N(p+q)-4s\max\{p,q\}-6N\big)}, & \big(p>\tfrac{3}{2},\,q<\tfrac{3}{2}\big)\;\text{or}\;\big(p<\tfrac{3}{2},\,q>\tfrac{3}{2}\big), \\[4pt]
\epsilon^{\tfrac{(N-2s)(9-2\max\{p,q\})(4s-N)}{8s}}\,|\log \epsilon|^{\frac{3(N-2s)}{2N}}, & \big(p>\tfrac{3}{2},\,q=\tfrac{3}{2}\big)\;\text{or}\;\big(p=\tfrac{3}{2},\,q>\tfrac{3}{2}\big).
\end{cases}
\end{cases}
\end{equation*}
 \end{itemize}
\end{lemma}
\begin{proof}
For parts \textit{(i)–(iii)}, the conclusions follow directly from Propositions 2.7 and 2.8 of  \cite{giacomoni2018doubly}. \textit{(iv)} is obtained by combining identity \eqref{AK2} with \cite[equation (4.9)]{he2022normalized}. Finally, the proof of $(v)$ is achieved by applying the results from \textit{(iii)} and \textit{(iv)}.
            
\end{proof} 
\begin{comment}

\section{Preliminaries}\label{S2}

\begin{definition}\label{ak1}
	\cite{wei2022normalized} Let $X$ be a topological space and $B$ be a closed subset of $X$. We say that a class $\mathcal{F}$ of compact subsets of $X$ is a homotopy-stable family with an extended boundary $B$ if, for any set $A$ in $\mathcal{F}$ and any $\eta \in C([0,1]\times X;X)$ satisfying $\eta(t,x)=x$ for all $(t,x)\in (\{0\}\times X)\cup ([0,1]\times B)$ we have that $\eta(\{1\}\times A)\in \mathcal{F}$.
\end{definition}

\begin{lemma}\label{a13} 
	\cite{wei2022normalized} Let $\kappa$ be a $C^1$ function on a complete connected $C^1-$ Finsler manifold $X$ (without boundary) and consider a homotopy-stable family $\mathcal{F}$ of compact subset of $X$ with a closed boundary $B$. Set $m=m(\kappa,\mathcal{F})$ and let $F$ be closed subset of $X$ satisfying 
	\begin{itemize}
		\item[(i)]  $(A\cap F)\backslash B\ne \emptyset$ for every $A\in \mathcal{F}$,
		\item[(ii)] $\sup \kappa (B)\le m\le \inf \kappa(F).$
	\end{itemize} 
Then, for any sequence of sets $\{A_n\}_{n\in N}$ in $\mathcal{F}$ such that $\displaystyle\lim_{n\to\infty}\sup_{A_n} \kappa=m$, there exists a sequence $\{x_n\}$ in $X$ such that
$$\lim\limits_{n\to\infty}\kappa(x_n)=m,~~\lim\limits_{n\to\infty}\|d\kappa(x_n)\|=0,~~\lim\limits_{n\to\infty}dist(x_n,F)=0,~~\lim\limits_{n\to\infty}dist(x_n,A_n)=0.$$ 
\end{lemma}	
\end{comment}
\maketitle
\section{Technical Results}\label{S3}
In this section, we discuss the Palais–Smale sequences and conduct a compactness analysis. To address the challenges posed by the nonlocal terms arising from the Kirchhoff operator, we construct a suitably perturbed Pohozaev manifold.
\qquad
\begin{lemma}\label{a1}\cite{moroz2013groundstates} Let $(u_n,v_n) \rightharpoonup (u,v)$ in $\h$, for $2_{\mu,*}<p,q<\2$. Then 
\begin{equation*}
\cnq=\cpq +o_n(1).
\end{equation*}
\end{lemma}
\begin{comment}

\begin{lemma}\label{a2}\cite[Lemma 2.7]{li2020normalized}
	If $v\in H^s(\R^N)\cap L^r(\R^N)$ with $N>2s, r\in (0,\frac{N}{N-2s}]$, and it satisfies
	\begin{equation*}
	\begin{cases}
	\tri^sv \ge 0,\\
	v\ge 0,
	\end{cases}
	\end{equation*}
then $v=0$.
\end{lemma}
\end{comment}
\begin{proposition}\label{a3}
	Assume $\2=3$, and either
   $ 2\tw<p+q<4+\frac{4s-2\mu}{N}$ or $\frac{8s-2\mu}{N}<p+q<6=2\2.$
  Also consider, a sequence $\{(u_n,v_n)\} \subset S_r(d_1,d_2)$, that forms a Palais-Smale sequence for the functional $\J$, restricted to the set ${S\D}$ at energy level $c \ne 0$ such that 
	\begin{equation*}
	c<\frac{ab\shl^3}{2}+\frac{b^3\shl^6}{12}+\frac{2}{3}\left(\frac{b^2\shl^4}{4}+a\shl \right)^{\frac{3}{2}},
	\end{equation*} 
\begin{equation*}
	\Pa(u_n,v_n)\to 0 ~~as~~ n\to \infty,
	\end{equation*}
	and $u_n^- \to 0,v_n^- \to 0$ a.e. in $\R^N$.
    
	Then, up to a subsequence, one of the following alternatives must hold:
	\begin{itemize}
		\item[(i)] The sequence $\{\uvn$\} converges weakly, but not strongly, in $\h$ to some nontrivial limit $(u,v)$, that satisfies the system   
		\begin{equation}\label{A_1}
		\begin{cases}
		(a+bD_1)\tri^{s}u=\lambda_1u+(I_{\mu}*|v|^{\2})|u|^{\2-2}u +\alpha p (I_{\mu}*|v|^{q})|u|^{p-2}u  ~~\text{in}~\mathbb{R}^N,\\
		(a+bD_2)\tri^{s}v=\lambda_2v+
		(I_{\mu}*|u|^{\2})|v|^{\2-2}u +\alpha q(I_{\mu}*|u|^{p})|v|^{q-2}u ~~\text{in}~\mathbb{R}^N,
		\end{cases}
		\end{equation}
		for some $\lambda_1,\lambda_2 <0$, $D_1:=\lim\limits_{n\to \infty}[u_n]_s^2 >0$ and $D_2:= \lim\limits_{n\to \infty}[v_n]_s^2 >0.$ Moreover, the limits satisfy the inequality
        \begin{equation*}
            c-\frac{ab\shl^3}{2}+\frac{b^3\shl^6}{12}+\frac{2}{3}\left(\frac{b^2\shl^4}{4}+a\shl \right)^{\frac{3}{2}}  \ge E_{\alpha}(u,v),
        \end{equation*}
        where
        \begin{equation*}
        E_{\alpha}(u,v):=\left(\frac{a}{2}+\frac{D_1b}{4}\right)[u]_s^2+\left(\frac{a}{2}+\frac{bD_2}{4}\right)[v]_s^2-\frac{1}{3}\cq-\alpha \cpq.
        \end{equation*} 
		\item [(ii)]  The sequence $\{\uvn\}$ converges strongly in $\mathcal{H}$ to some $(u, v) \in \mathcal{H}$, and in this case, $(u, v) \in S(d_1, d_2)$, $\J(u, v) = c$, and $(u, v)$ is a solution of the system \eqref{G} for some $\lambda_1, \lambda_2 < 0$.
	\end{itemize}
	
\end{proposition}
\begin{proof} The proof is divided into five steps.
	
	\textbf{Step 1:} $\{(u_n,v_v)\}$ is bounded in $\h$\par 
    
If $4+\frac{8s-2\mu}{N}< p+q<6$, then $\delta_p + \delta_q > 4$.
	Given that $\Pa(u_n,v_n)\to 0$ so, we have
	\begin{equation*}
	c+o_n(1)=\frac{a}{4}([u_n]_s^2+[v_n]_s^2)+\frac{1}{6}\cqn+\alpha\left(\frac{\dpq}{4}-1\right)\cnq.
	\end{equation*}
	
	Since $\delta_p + \delta_q \ge 4$, we get $\int_{\mathbb{R}^N}(I_{\mu}*|u_n|^p)|v_n|^q$, $\int_{\mathbb{R}^N}(I_{\mu}*|u_n|^3)|v_n|^3$ and $\frac{a}{4}([u_n]_s^2+[v_n]_s^2)$ are bounded. It results that, $\{(u_n,v_n)\}$ is bounded in $\h$.\par 
	
	If $2\tw<p+q<4+\frac{4s-2\mu}{N}$, then $0<\delta_p + \delta_q <2$. We have
	$$\frac{a}{3}([u_n]_s^2+[v_n]_s^2)+\frac{b}{12}([u_n]_s^4+[v_n]_s^4)=c+o_n(1)+\alpha\left(1-\frac{\dpq}{6}\right)\cnq.$$
	From \eqref{Ai4} it follows that
	\begin{align*}
	\frac{a}{3}([u_n]_s^2+[v_n]_s^2)+\frac{b}{12}([u_n]^4_s+[v_n]^4_s)\le c+o_n(1)
	+\alpha\left(1-\frac{\dpq}{6}\right) B_{N,p,q,s,\mu}(d_1^2+d_2^2)^{\frac{p+q-(\dpq)}{2}}([u_n]_s^2+[v_n]_s^2)^{\frac{\dpq}{2}},
	\end{align*}
	which gives $([u_n]_s^2+[v_n]_s^2)$ is bounded. Consequently, the sequence $\{(u_n,v_n)\}$ is bounded in $\h$.\par 
\textbf{Step 2:} There exist Lagrange multipliers $\lambda_{1,n},$ $ \lambda_{2,n}$\par 	Since $\{(u_n,v_n)\}$ forms a Palais-Smale sequence of $\J|_{S\D}$, according to the Lagrange's multipliers rule, there exists $\lambda_{1,n},$ $\lambda_{2,n} \in \R$, such that
	\begin{align}\label{A4}
	\notag\lambda_{1,n}\int_{\R^N}u_n\phi+\lambda_{2,n}\int_{\R^N}v_n\xi +o_n(1)=& (a+b[u_n]_s^2)\Aa(u_n,\phi)+(a+b[v_n]_s^2)\Aa(\xi,v_n) -\int_{\R^N}(I_{\mu}*|v_n|^3)|u_n|u_n\phi\\ \notag&-\int_{\R^N}(I_{\mu}*|u_n|^3)|v_n|v_n\xi
	-\alpha p \int_{\R^N}(I*|v_n|^q)|u_n|^{p-2}u_n\phi\\& -\alpha q\int_{\R^N}(I_{\mu}*|u_n|^p)|v_n|^{q-2}v_n\xi ,
	\end{align}
	for every $(\phi ,\xi) \in \h$, where $o_n(1)\to 0$ as $n \to \infty$. For $(\phi,\xi)=(u_n,v_n)$, we have
	\begin{align}\label{A5}
\lambda_{1,n}d_1^2+\lambda_{2,n}d^2_2+
o_n(1)= a([u_n]_s^2+[v_n]_s^2)+b([u_n]_s^4+[v_n]_s^4)-2\int_{\R^N}(I_{\mu}*|u_n|^3)|v_n|^3 	-\alpha (p+q) \int_{\R^N}(I*|v_n|^q)|u_n|^{p}.
\end{align}
By \eqref{A1} and \eqref{A5}, it follows that
\begin{equation*}
\lambda_{1,n}d_1^2+\lambda_{2,n}d^2_2=-\alpha(p+q-(\dpq)) \int_{\R^N}(I*|v_n|^q)|u_n|^{p}+o_n(1).
\end{equation*}
Since, $\{\uvn\}$ is bounded in $\h$, therefore both $\{\lambda_{1,n}\}$ and $\{\lambda_{2,n} \}$ are also  bounded in $\R$. Thus, there exist $(u, v)\in \h_r$, and $ \lambda_1, \lambda_2 \in \R$, such that $(u_n, v_n)  \rightharpoonup  (u, v) $ in $\h$, $(u_n,v_n) \to (u, v)$ a.e. in $\R^{2N}$ and $(\lambda_{1,n},\lambda_{2,n}) \to (\lambda_1,\lambda_2 )$ in $\R^2$. By the fact that $u_n^-\to 0,v_n^-\to 0$ a.e. in $\R^N$, we conclude  $u\ge 0,v\ge 0$. 

\textbf{Step 3:} $u\not\equiv0$ and $v\not\equiv 0 $\par
	Suppose, if possible, $u\equiv0$ or $v\equiv0$. Then, by Lemma \ref{a1}, we have 
    $$\int_{\R^N}(I_{\mu}*|v_n|^q)|u_n|^{p}\to 0.$$
    Therefore, we obtain
\begin{equation}\label{A7}
\J(u_n,v_n)=\frac{a}{2}([u_n]_s^2+[v_n]_s^2)+\frac{b}{4}([u_n]_s^4+[v_n]_s^4)
-\frac{1}{3}\cqn=c+o_n(1),
\end{equation}	
and 
\begin{equation}\label{A8}
 \Pa(u_n,v_n)=a([u_n]_s^2 + [v_n]_s^2) +b([u_n]_s^4+[v_n]_s^4) -2\cqn 
=o_n(1).
\end{equation}	
Let $A,B \ge 0$ be such that 
\begin{equation}\label{A9}
a([u_n]_s^2+[v_n]_s^2) \to A~\text{as}~n \to \infty
\end{equation}
and
\begin{equation}\label{A10}
b([u_n]_s^4+[v_n]_s^4) \to B~\text{as}~n \to \infty.
\end{equation}
Then $A$ and $B$ must be greater than 0, because if $A=0$ or $B=0$, we have $u_n\to 0$ and $v_n\to 0$ in $H^s(\R^N)$, which is a contradiction to our assumption that $c\ne0$.
From \eqref{A8}, we obtain that
\begin{equation}\label{A11}
\cqn \to \frac{A+B}{2}.
\end{equation}
Using \eqref{A9}, \eqref{A10} and \eqref{A11} in \eqref{A7} and taking $n\to \infty$, we have
\begin{equation}\label{A12}
\frac{A}{3}+\frac{B}{12}=c.
\end{equation}

Using \eqref{Ai2} and \eqref{A9}, one gets
\begin{align}\label{A13}
\notag \lim_{n\to\infty}a\shl\left(\int_{\R^N}(I_{\mu}*|u_n|^3)|v_n|^3\right)^{\frac{1}{3}}\le \frac{A}{2}\\
a\shl\left(\frac{A+B}{2} \right)^{\frac{1}{3}} \le \frac{A}{2}.
\end{align}
Similarly, using \eqref{Ai2} and \eqref{A10}, we have
\begin{equation}\label{A14}
b\shl^2\left(\frac{A+B}{2} \right)^{\frac{2}{3}} \le \frac{B}{2}.
\end{equation}
By \eqref{A13} and \eqref{A14} with $\zeta:=\left(\frac{A+B}{2} \right)^{\frac{1}{3}}$, we obtain
\begin{equation*}
\zeta^3-b\shl^2\zeta^2-a\shl\zeta \ge 0.
\end{equation*}
This gives us 
\begin{equation}\label{A16}
\zeta\ge\frac{b\shl^2}{2}+\sqrt{\frac{b^2\shl^4}{4}+a\shl}.
\end{equation}
Using \eqref{A12}-\eqref{A16}, we obtain
\begin{align*}
c&\ge \frac{2a\shl}{3}\left(\frac{A+B}{2} \right)^{\frac{1}{3}}+\frac{b\shl^2}{6}\left(\frac{A+B}{2} \right)^{\frac{2}{3}}\\
&=\frac{2a\shl}{3}\zeta+\frac{b\shl^2}{6}\zeta^2\\
&\ge \frac{ab\shl^3}{2}+\frac{b^3\shl^6}{12}+\frac{2}{3}\left(\frac{b^2\shl^4}{4}+a\shl \right)^{\frac{3}{2}}.
\end{align*}
This leads to a contradiction with our assumption on 
$c$, from which it follows that 
$u\not\equiv0$ and $v\not\equiv0$.
\par 
 
\textbf{Step 4:} $\lambda_1,\lambda_2 <0$.\par
 Suppose, if possible, $\lambda_1 \ge 0$. Since  $\uvn \rightharpoonup (u,v)$ weakly in $\h$, therefore we have
 \begin{align*}
 (a+bD_1)\tri^{s}u&=\lambda_1u+(I_{\mu}*|v|^{\2})|u|^{\2-2}u +\alpha p (I_{\mu}*|v|^{q})|u|^{p-2}u  ~~\text{in}~\mathbb{R}^N\\
 &\ge0.
 \end{align*}
Using the maximum principle \cite[Proposition 2.17]{silvestre2007regularity}, we have $u=0$, a contradiction.
So, we have $\lambda_1<0$. Similarly, $\lambda_2<0.$\par 

\textbf{Step 5:} Conclusion.\par 
 Since $\uvn \rightharpoonup (u,v)$ in $\h$ and $u\not\equiv 0,v\not\equiv 0$, $D_1:=\lim\limits_{n\to\infty}[u_n]_s^2 \ge [u]_s^2>0$ and $D_2:=
\lim\limits_{n \to \infty}[v_n]_s^2 \ge [v]_s^2 >0.$ Then, using \eqref{A4}, for every $(\phi,\xi) \in \h$, we get
\begin{align}\label{A6}
\notag (a+D_1b)\Aa(u,\phi)+(a+D_2b)\Aa(v,\xi) &-\int_{\R^N}(I_{\mu}*|v|^3)|u|u\phi-\int_{\R^N}(I_{\mu}*|u|^3)|v|v\xi
-p\alpha \int_{\R^N}(I*|v|^q)|u|^{p-2}u\phi\\&-q\alpha\int_{\R^N}(I_{\mu}*|u|^p)|v|^{q-2}v\xi =0.
\end{align}
That is, $(u,v)$ satisfies 
\begin{equation}\label{A50}
\begin{cases}
(a+bD_1)\tri^su=\lambda_1u+(I_{\mu}*|v|^{3})|u|u +\alpha p (I_{\mu}*|v|^{q})|u|^{p-2}u ) ~~\text{in}~\mathbb{R}^N,\\
(a+bD_2)\tri^2v=\lambda_2v+
(I_{\mu}*|u|^{3})|v|v +\alpha q(I_{\mu}*|u|^{p})|v|^{q-2}v )~~\text{in}~\mathbb{R}^N.
\end{cases}
\end{equation}
So the Pohozaev identity, corresponding to \eqref{A50}, is
\begin{align*}
\tilde{P}_{\alpha}(u,v):= (a+D_1b)[u]_s^2+(a+D_2b)[v]_s^2-2\int_{\R^N}(I_{\mu}*|v|^3)|u|^3-\alpha(\dpq)\cpq=0.
\end{align*}
Let  $\tilde{u}_n=u_n-u, ~\tilde{v}_n=v_n-v$, then $(\tilde{u}_n,\tilde{v}_n) \rightharpoonup (0,0)$  in $\h$ and $[u_n]_s^2=[u]_s^2+[\tilde{u}_n]_s^2+o_n(1)$, $[v_n]_s^2=[v]_s^2+[\tilde{v}_n]_s^2+o_n(1)$. By Brezis-Lieb lemma \cite{brezis1983relation}, we have
$$\int_{\R^N}(I_{\mu}*|u_n|^3)|v_n|^3=\cqt+\tcqn+o_n(1).$$
%$$ \cpqn=\cpq+\cspq+o_n(1).$$\par 
Since $\int_{\R^N}(I_{\mu}*|\tilde{u}_n|^p)|\tilde{v}_n|^q \to 0$ strongly, we have $\cnq=\cpq +o_n(1).$ So, $\Pa(u_n,v_n)$ can be written as
\begin{align*}{P}_{\alpha}(u_n,v_n)=(a+D_1b)[u_n]_s^2+(a+D_2b)[v_n]_s^2-\alpha(\dpq)\cnq-2 \cqn+o_n(1).
\end{align*}
From $\tilde{P}_{\alpha}(u,v)=0$, we obtain 
\begin{align}\label{A17}
\notag \lim\limits_{n\to\infty}2\tcqn&=\lim\limits_{n \to \infty}\left((a+bD_1)[\tilde{u}_n]_s^2+(a+bD_2)[\tilde{v}_n]_s^2 \right)\\
&\ge \lim\limits_{n \to \infty} \left( a([\tilde{u}_n]_s^2+[\tilde{v}_n]_s^2)+b([\tilde{u}_n]_s^4+[\tilde{v}_n]_s^4)\right).
\end{align}
Let $\tilde{A}:= a([\tilde{u}_n]_s^2+[\tilde{v}_n]_s^2)$ and $\tilde{B}:=b([\tilde{u}_n]_s^4+[\tilde{v}_n]_s^4)$. Using \eqref{Ai2} and \eqref{A17}, we deduce 
\begin{equation*}
a\shl\left(\frac{\tilde{A}+\tilde{B}}{2} \right)^{\frac{1}{3}}+b\shl^2\left(\frac{\tilde{A}+\tilde{B}}{2} \right)^{\frac{2}{3}}\le \left(\frac{\tilde{A}+\tilde{B}}{2} \right).
\end{equation*}
Consider $\tilde{\zeta}:=\left(\frac{A+B}{2} \right)^{\frac{1}{3}}$. Then, we have 
$$\tilde{\zeta}^3\ge a\shl\tilde{\zeta}+b\shl^2\tilde{\zeta}^2$$ 
and
\begin{align}
\notag \lim\limits_{n\to \infty} \left(\frac{\tilde{A}+\tilde{B}}{2} \right) \le \lim\limits_{n\to\infty} \int_{\R^N}(I_{\mu}*|\tilde{u}_n|^3)|\tilde{v}_n|^3 \le \frac{1}{\shl^3}\lim_{n\to\infty}\left(\frac{[\tilde{u}_n]_s^2+[\tilde{v}_n]_s^2}{2}\right)^3.
\end{align}
We get $\displaystyle\tcqn \ge {\tilde{\zeta}}^3$ and $\lim\limits_{n\to \infty}\left([\tilde{u}_n]_s^2+[\tilde{v}_n]_s^2\right)\ge 2\shl\tilde{\zeta}$ (where $\tilde{\zeta}\ge\frac{b\shl^2}{2}+\sqrt{\frac{b^2\shl^4}{4}+a\shl}$) or $\displaystyle\tcqn =0=\lim\limits_{n\to\infty}\left([\tilde{u}_n]_s^2+[\tilde{v}_n]_s^2\right)$.

\begin{itemize}
	\item[\textit{(i)}] If $\tcqn \ge \tilde{\zeta}^3$, then we have:
	\begin{align*}
c&=\lim\limits_{n\to\infty}\J(u_n,v_n)\\&=E_{\alpha}(u,v)+\lim\limits_{n\to\infty} \left(  \left(\frac{a}{2}+\frac{D_1b}{4}\right)[\tilde{u}_n]_s^2+\left(\frac{a}{2}+\frac{D_2b}{4}\right)[\tilde{v}_n]_s^2-\frac{1}{3}\tcqn \right)\\
	&\ge E_{\alpha}(u,v)+\lim\limits_{n\to\infty}\left(\frac{a}{3}\left([\tilde{u}_n]_s^2+[\tilde{v}_n]_s^2\right)+\frac{b}{12}\left([\tilde{u}_n]_s^4+[\tilde{v}_n]_s^4\right)\right)\\
	&\ge E_{\alpha}(u,v)+\frac{ab\shl^3}{2}+\frac{b^3\shl^6}{12}+\frac{2}{3}\left(\frac{b^2\shl^4}{4}+a\shl \right)^{\frac{3}{2}} 
	\end{align*} 
	where $E_{\alpha}(u,v):=(\frac{a}{2}+\frac{D_1b}{4})[u]_s^2+(\frac{a}{2}+\frac{bD_2}{4})[v]_s^2-\frac{1}{3}\int_{\R^N}(I_{\mu}*|v|^3)|u|^3-\alpha \cpq$. Alternative \textit{(i)} is satisfied in this case.  
	\item[\textit{(ii)}] If $\tcqn = 0$, then by relation \eqref{A17}, it follows that $[\tilde{u}]_s^2 = 0$ and $[\tilde{v}]_s^2 = 0$, which implies $(u_n,v_n) \to (u, v)$ in $D^s(\R^N)\times D^s(\R^N)$. 
     Testing \eqref{A4} and \eqref{A6} with $(\phi,\xi)=(u_n-u,v_n-v)$ yields
	$$(a+D_1b)[u_n-u]_s^2+(a+D_2b)[v_n-v]_s^2-\lambda_1\|u_n-u\|^2_2-\lambda_2\|v_n-v\|_2^2\to 0.$$ In this case, alternative \textit{(ii)} holds.
\end{itemize}
\end{proof}
\maketitle
\section{Proof of Theorem \ref{AAA}}\label{S4}
This section is devoted to the proof of Theorem~\ref{AAA}, which establishes the existence and asymptotic behavior of normalized ground states for problem \eqref{G}.  We begin by analyzing the Pohozaev manifold $\mathcal{P}\D$ and investigating the geometric properties of the associated energy functional.
\begin{lemma}\label{a5}
	Let $2\tw <p+q<4+\frac{4s-2\mu}{N}$ and $0<\alpha <\alpha^*$ then, in $\h,~ \mathcal{P}\D $ is a smooth submanifold and $\mathcal{P}^{0}\D =\emptyset$.
\end{lemma}	
\begin{proof}
First, we assert that $\mathcal{P}^{0}\D =\emptyset$. In this absence, $(u,v)\in \mathcal{P}^{0}\D $ exists. Given $\Pa(u,v) = 0$ and $(\f)''(0) = 0$, we obtain
\begin{equation*}
		a([u]_s^2+[v]_s^2)+b([u]_s^4+[v]_s^4)=2\cq+\alpha(\dpq)\cpq,
		\end{equation*}
		and
		\begin{equation*}
		2a([u]_s^2 +[v]_s^2)+4b([u]_s^4+[v]_s^4)=\2 4\cq+\alpha(\dpq)^2\cpq.
		\end{equation*}
		By solving the above equation and using \eqref{Ai2}, \eqref{Ai4} we have
		\begin{align*}
	 a(2-(\dpq))([u]_s^2+[v]_s^2)+b(4-(\dpq))([u]_s^4+[v]_s^4)
     &=2(2\2-(\dpq))\cq\\
	 &\le 2(2\2-(\dpq))\left(\frac{[u]_s^2+[v]_s^2}{2\shl}   \right)^{\2},
		\end{align*}
		and
		\begin{align*}
		4a([u]_s^2+[v]_s^2)+2b([u]_s^4+[v]_s^4)&=\alpha(\dpq)(2\2-(\dpq))\cpq\\
	&\le \alpha(\dpq)(2\2 -(\dpq))B_{N,p,q,s,\mu}(d_1^2+d_2^2)^{\frac{p+q-(\dpq)}{2}}([u]_s^2+[v]_s^2)^{\frac{\dpq}{2}}.
		\end{align*}
		Subsequently, lower and upper bounds  of $([u]_s^2+[v]_s^2)$ are provided by
		\begin{align*}
		([u]_s^2+[v]_s^2)\ge \left(\dfrac{b(4-(\dpq))}{4(2\2 -(\dpq))}(2\shl)^{\2}\right) ^{\frac{1}{\2-2}}
    	\end{align*}
	    \begin{align*}
		([u]_s^2+[v]_s^2) \le \Big( \dfrac{\alpha(\dpq)(2\2 -(\dpq))B_{N,p,q,s,\mu}(d_1^2+d_2^2)^{\frac{p+q-(\dpq)}{2}}}{b}\Big)^{\frac{2}{4-(\dpq)}}
		\end{align*}
        which implies
	\begin{align*}
		\alpha \ge \left(\dfrac{b(4-(\dpq))}{4(2\2 -(\dpq))}(2\shl)^{\2 }  \right)^{\frac{4-(\dpq)}{2(\2-2)}}\left( \dfrac{b}{(\dpq)(2\2 -(\dpq))B_{N,p,q,s,\mu}(d_1^2+d_2^2)^{\frac{p+q-(\dpq)}{2}}}\right)= \alpha^*,
	\end{align*}
		which is a contradiction to $\alpha < \alpha^*$. Hence, our assumption is wrong, consequently $\mathcal{P}^0\D=\emptyset
        .$ \\
		It is easy to verify that $\mathcal{P}\D$ is a sub-manifold of $\h$ of co-dimension 3, like in the proof of  \cite[Lemma 4.2]{hu2023normalized}.
	\end{proof}
 
%Considering that $\mathcal{P}^{0}\D =\emptyset,\mathcal{P}\D =\mathcal{P}^{+}\D \cup\mathcal{P}^{-}\D ,$ where $\mathcal{P}^{+}\D \cap\mathcal{P}^{-}\D =\emptyset$. In the sense that follows, we can demonstrate that $\mathcal{P}\D$ is a natural constraint:

\begin{lemma}
	Assume $2\tw <p+q <4+\frac{4s-2\mu}{N}$ and $\alpha \in(0,\alpha_{*})$. If $(u,v)\in \mathcal{P}(d_1,d_2)$ is a critical point for $\J|_{\mathcal{P}\D} $, then it also becomes a critical point for $\J|_{S\D}$. 
\end{lemma}
\begin{proof}
	If $(u,v)\in \mathcal{P}\D$ is a critical point for $\J|_{\mathcal{P}\D}$, by the Lagrange multiplier's rule, we obtain that there exists $\lambda_1,\lambda_2,\gamma \in \R$ such that for every
	$(\phi,\xi)\in \h$, we have 
	\begin{equation*}
	\langle (\J)'(u,v)(\phi,\xi)\rangle =\lambda_1\int_{\R^N}u\phi+\lambda_2\int_{\R^N}v\xi+ \gamma\langle(\mathcal{P}\D)'(u,v)(\phi,\xi)\rangle.
	\end{equation*}
 This means that $(u,v)$ solves the following equation
 \begin{align}\label{A66}
 \begin{cases}
\big((1-2\gamma)a+(1-4\gamma)b\int_{\R^N}|\tri^{\s}u|^2\big)\tri^su=&\lambda_1u+\alpha p(1-\gamma(\dpq))\int_{\R^N}(I_{\mu}*|v|^{q})|u|^{p-2}u\\
&+(1-2\2\gamma )\int_{\R^N}(I_{\mu}*|v|^{\2})|u|^{\2-2}u)\ \text{in}~\R^N,\\
  \big((1-2\gamma)a+(1-4\gamma)b\int_{\R^N}|\tri^{\s}v|^2\big)\tri^sv=&\lambda_2v
  +\alpha q(1-\gamma(\dpq))\int_{\R^N}(I_{\mu}*|u|^{p})|v|^{q-2}v\\&+(1-2\2\gamma )\int_{\R^N}(I_{\mu}*|u|^{\2})|v|^{\2-2}v)\ \text{in}~\R^N.
 \end{cases}
 \end{align}
 Pohozaev identity corresponding to \eqref{A66}, is
\begin{align}\label{A19}
\notag a(1-2\gamma)([u]_s^2 + [v]_s^2) +b(1-4\gamma)([u]_s^4+[v]_s^4) -\alpha(\delta_p +\delta_q)(1-\gamma(\dpq ))\cpq\\
-2(1-2\2\gamma )\cq =0.
\end{align}
Since $(u,v)\in \mathcal{P}\D$, and from \eqref{A19}, we have
\begin{align*}
\gamma\big(2a([u]_s^2+[v]_s^2)+4b([u]_s^4+[v]_s^4)-\24\cq -\alpha(\dpq)^2\cpq\big)=0.
\end{align*}
 From Lemma \ref{a5}, we deduce that $\mathcal{P}^{0}\D =\emptyset$. Hence, $\gamma=0$, which proves the lemma.
\end{proof}
To find the locations of critical points for $\J |_{S\D}$, with the of help \eqref{Ai2}, \eqref{Ai4}, we have
	\begin{align}\label{A20}
\J(u,v) &\ge \frac{a}{2}\Big([u]_s^2+[v]_s^2  \Big)+ \frac{b}{8}\Big([u]_s^2+[v]_s^2  \Big)^2-\frac{1}{\2(2\shl)^{\2 }} \Big([u]_s^2+[v]_s^2  \Big)^{\2}
\\
&-\alpha B_{N,p,q,s,\mu}(d_1^2+d_2^2)^{\frac{p+q-(\dpq)}{2}}\Big([u]_s^2+[v]_s^2  \Big)^{\frac{\dpq}{2}},
	\end{align}
$\forall(u,v)\in \h.$ For understanding the geometry of $\J|_{S\D}$, we introduce the function $g:\R^+\to \R$:
\begin{equation*}
g(t)=\frac{a}{2}t^2+\frac{b}{8}t^4-\frac{1}{\2(2\shl)^{\2 }}t^{2\2}-\alpha B_{N,p,q,s,\mu}(d_1^2+d_2^2)^{\frac{p+q-(\dpq)}{2}}t^{\dpq},
\end{equation*}	
where $(\dpq)<2$ and $\2>2$. We have that $g(0^+)=0^-$ and $g(+\infty)=-\infty$.
	\begin{lemma}\label{a7} \cite[Lemma 4.3]{li2022normalized}
	Let $\tilde{a},\tilde{b},\tilde{c},\tilde{d},\tilde{p},\tilde{q}>0$
	and 
	$f(t):=\tilde{a}t^2+\tilde{b}^4-\tilde{c}t^{\tilde{p}}-\tilde{d}t^{\tilde{q}}$ for $t\ge0$. If $\tilde{p}>4, \tilde{q}\in (0,2)$ and 
    \begin{equation}\label{A101}\left( \left(\frac{8(4-\tilde{q})}{\tilde{p}(\tilde{p}-2)(\tilde{p}-\tilde{q})}\right)^{\frac{4-\tilde{q}}{\tilde{p}-4}}-\left(\frac{8(4-\tilde{q})}{\tilde{p}(\tilde{p}-2)(\tilde{p}-\tilde{q})}    \right)^{\frac{\tilde{p}-\tilde{q}}{\tilde{p}-4}}   \right)\left(\frac{\tilde{a}}{\tilde{d}}\left( \frac{\tilde{b}}{\tilde{c}}  \right)^{\frac{2-\tilde{q}}{\tilde{p}-4}}+\frac{1}{\tilde{d}}\frac{\tilde{b}^{\frac{\tilde{p}-\tilde{q}}{\tilde{p}-4 }} }{c ^{\frac{4-\tilde{q}}{ \tilde{p}-4}  }}  \right)  >1.
    \end{equation}
    Then, $f(t)$ has a local strict minimum at the negative level and a global strict maximum at the positive level on $[0,+\infty)$.             
	\end{lemma}	
\begin{lemma}\label{a30}
		Assume $2\tw <p+q <4+\frac{4s-2\mu}{N}$ and $\alpha\in(0,\alpha_{*})$ then, the function $g$ has a local minimum at the negative level and a global strict maximum at the positive level. Moreover, depending on $d_1, d_2$ and $\mu$, there exist $0<R_*<R^*$, such that $g(R_*)=g(R^*)=0$ and $g(t)>0$, if and only if $t\in (R_*,R^*)$.
\end{lemma}
\begin{proof}
	Take $\tilde{a}=\frac{a}{2},~ \tilde{b}=\frac{b}{4}, ~\tilde{c}=\frac{1}{\2(2\shl)^{\2}}, ~\tilde{d}=\alpha B_{N,p,q,s,\mu}(d_1^2+d_2^2)^{\frac{p+q-(\dpq)}{2}},~ \tilde{q}=\dpq$ and $ \tilde{p}=2\2$ in Lemma \ref{a7}, thus the conclusion follows provided $0<\alpha <\alpha_*$.
\end{proof}	
	\begin{lemma}\label{a9}
		Let $0<\alpha<min\{\alpha^*,\alpha_*\}$. For every $u\in S\D$, the function $\f$ has exactly two critcal points $k_1<k_2 \in \R $ and two zeros $\ell_1<\ell_2 \in  \R$ with $k_1<\ell_1<k_2<\ell_2$. Moreover:
		\begin{itemize}
			\item[(i)] $k_1\star (u,v)\in \mathcal{P}^{+}\D $ and $k_2 \star (u,v)\in \mathcal{P} ^-\D$ and if $k\star(u,v) \in \mathcal{P}\D $ then either $k=k_1$ or $k=k_2$;
			\item[(ii)] $\left([k\star u]_s^2+[k\star v]_s^2\right)^{\frac{1}{2}} \le R_{*}$ for every $k\le l_1$ and
			$$\J(k_1\star(u,v))= \min\{\J(k\star (u,v)):k\in \R ~\text{and}~\left([k\star u]_s^2+[k\star v]_s^2\right)^{\frac{1}{2}} \le R_{*}\}<0,$$
			\item[(iii)] We have
			$$\J(k_2\star(u,v)) =\min\{\J(k\star(u,v)):k\in \R \}>0,$$
			and $\F$ is strictly decreasing on $(k_2,\infty)$;
			\item[(iv)] The map $(u,v)\in S\D \mapsto k_1 \in \R$ and $(u,v)\in S\D \mapsto k_2 \in \R$ are of class $\mathcal{C}^1$.
		\end{itemize}
	\end{lemma}
\begin{proof}
For $(u,v)\in S\D$, we have 
\begin{equation*}
\begin{aligned}
     \J(k \star (u,v))=\frac{a}{2}e^{2sk}([u]_s^2 +[v]_s^2) +\frac{b}{4}e^{4sk}([u]_s^4 + [v]_s^4)
-\frac{1}{\2}e^{2s\2 k}\cq \\-\alpha e^{s(\delta_p + \delta_q)k}\cpq.
\end{aligned}
\end{equation*}
If $\tilde{a}=\frac{a}{2}([u]_s^2 +[v]_s^2),\tilde{b}=\frac{b}{4}([u]_s^4 + [v]_s^4), \ \tilde{c}=\frac{1}{\2}\cq, \ \tilde{d}=\alpha \cpq, \ \tilde{p}=2\2$ and $\tilde{q}=\dpq$, then it satisfies condition \eqref{A101}.
	From \eqref{A20}, we obtain
	\begin{align*}
    \J(k\star(u,v)) &\ge g\left(([k\star u]_s^2+[k\star v]_s^2 )^{\frac{1}{2}}\right)\\
    &=g(e^{ks}([u]_s^2+[v]_s^2)^{\frac{1}{2}})
	\end{align*}
	then
	\begin{equation*}
	\f(k)>0,~~\forall ~k\in \left(\log\frac{R_*}{([u]_s^2+[v]_s^2)^{\frac{1}{2}}},\log\frac{R^*}{([u]_s^2+[v]_s^2)^{\frac{1}{2}}}  \right).
	\end{equation*}
    	As $\f(-\infty)=0^-$ and $\f(+\infty)=-\infty$, we see that $\f$ has exactly two critical points $k_1<k_2$ with $k_1$ as local minimum point on $\left(-\infty,\log\frac{R_*}{([u]_s^2+[v]_s^2)^{\frac{1}{2}}}  \right)$ at negative level and $k_2$ as global maximum point at a positive level. Using Corollary \ref{a0}, we have $(k_1\star(u,v)),(k_2\star(u,v))\in \mathcal{P}\D,$ and $(k\star(u,v))\in \mathcal{P}\D$ which implies $k\in \{k_1,k_2\}$. By minimality $(\f)''(0)=(\f)''(k_1)\ge 0$, and ``$=$" does not hold, as $\mathcal{P}^0\D =\emptyset$; so $(k_1\star(u,v)) \in \mathcal{P}^{+}\D .$ Similarly, we have $(k_2 \star(u,v))\in \mathcal{P}^{-}\D $. Because of the behavior at infinity and monotonicity, $\f$ has exactly two zeros $\ell_1<\ell_2$ with $k_1<\ell_1<k_2<\ell_2$.\par 
	Consider the $C^1$ function $\Psi(s,u,v):=(\f)'(k)$. By the fact that $\Psi(k_1,u,v)=0,~\partial_k\Psi(k_1,u,v)>0$, and it is impossible to pass with continuity from $\mathcal{P}^{+}\D $ to $\mathcal{P}^{-}\D $ (as $\mathcal{P}^{0}\D =\emptyset$), then by applying implicit function theorem on $\Psi(k,u,v)$, we get the desired result. 
\end{proof}	
	For $z>0$, define 
	$$\mathcal{B}_z:= \{(u,v)\in S\D :([u]_s^2+[v]_s^2)^{\frac{1}{2}} < z  \},$$ 
	and
	$$m_{\alpha}\D:= \inf_{u\in \mathcal{B}_{R_*}}\J(u,v).$$	
\begin{corollary}\label{a8}
		Let $s\in(0,1),~22_{\mu,*} < p+q <4+\frac{4s-2\mu}{N}$ and $0<\alpha<min\{\alpha^*,\alpha_*\}$. Then $\mathcal{P}^{+}\D  \subset \mathcal{B}_{R^*}$ and $\sup\limits_{\mathcal{P}^{+}\D }\J(u,v) \le 0\le \inf\limits_{\mathcal{P}^{-}\D }\J(u,v).$ 
\end{corollary}	
\begin{proof}
For all $(u,v)\in \mathcal{P}^{+}\D $, Lemma \ref{a9} implies that $k_1=0,~\J(u,v)\le 0$ and ${([u]_s^2+[v]_s^2)^{\frac{1}{2}}}<R_*$. Similarly, $(u,v)\in \mathcal{P}^{-}\D$ implies $k_2=0$ and $\J(u,v)\ge 0$.

\end{proof}

	\begin{corollary}\label{a10}
		Let $s\in(0,1), ~22_{\mu,*} < p+q <4+\frac{4s-2\mu}{N}$ and $0<\alpha<min\{\alpha^*,\alpha_*\}$. Then, $m_{\alpha}\D \in (-\infty,0)$ and 
		$$m_{\alpha}\D=\displaystyle\inf_{\mathcal{P}\D }\J=\inf_{\mathcal{P}^{+}\D }\J .$$
		Also, there exists a constant $\rho >0$ small enough such that 
		$m_{\alpha}\D<\inf_{\overline{\mathcal{B}}_{R_*}\backslash \mathcal{B}_{R_*-\rho}}\J.$ Here, $\overline{\mathcal{B}}_{R_*}$ denotes closure of $\mathcal{B}_{R_*}$.
	\end{corollary}	
	\begin{proof}
		For $(u,v) \in \mathcal{B}_{R_*},$ we have $\J(u,v) \ge g\left(([u]_s^2+[v]_s^2)^{\frac{1}{2}}\right) \ge \min\limits_{t\in [0,R_*]}g(t)>-\infty$, and then $m_{\alpha}\D >-\infty$. Furthermore, for any $(u,v) \in S\D$, we have $([u]_s^2+[v]_s^2)^{\frac{1}{2}}<R_*$ and $\J(t\star(u,v))<0$ for $t<<-1$ and hence $m_{\alpha}\D <0$.\par 
		By Corollary \ref{a8}, we deduce $m_{\alpha}\D \le \inf\limits_{\mathcal{P}^{+}\D } \J$ as $\mathcal{P}^{+}\D \subset \mathcal{B}_{R_*}$. On the other hand, if we have $k_2\star (u,v) \in \mathcal{P}^{+}\D  \subset \mathcal{B}_{R_*}$, and  
		\begin{align*}
		\J(k_2\star(u,v)) =\min\{\J(k\star(u,v)):k\in \R ~and~ ([k\star u]_s^2+[k\star v]_s^2)^{\frac{1}{2}} <R_*  \} \le \J(u,v),
		\end{align*}
which implies that $\inf\limits_{\mathcal{P}^{0}\D }\J \le m_{\alpha}\D.$ Since $\J(u,v) \ge 0$ on $\mathcal{P}^{-}\D $ and we get $\inf\limits_{\mathcal{P}^{}\D } \le m_{\alpha}\D.$ Since $\J \ge 0$ on $\mathcal{P}^{-}\D $, we get $\inf\limits_{\mathcal{P}^{+}\D }\J=\inf\limits_{\mathcal{P}\D}\J$. Then, by continuity of $g$ there exists $\rho>0$ (independent of $d_1,d_2$ and $\alpha$), such that $g(t) \ge \frac{m_{\alpha}\D}{2}$, if $t \in [R_*-\rho , R_*]$. Hence 
		 \begin{align*}
		 	\J(u,v)\ge g(([u]_s^2+[v]_s^2)^{\frac{1}{2}})\ge \frac{m_{\alpha}\D}{2}>m_{\alpha}\D, 
		 \end{align*}
		 for every $(u,v) \in \overline{\mathcal{B}}_{R_*}\backslash \mathcal{B}_{R_*-\rho}.$			 
	\end{proof}

\textbf{Proof of Theorem \ref{AAA}:}
\textit{(i)} Let $\{(u_n,v_n)\}$ be minimizing sequence for $m_{\alpha}\D=\inf\limits_{u\in \mathcal{B}_{R_*}}\J$. Using \cite{lieb2001analysis,park2011fractional,servadei2013variational}, we infer the inequality $\J(|u_n|^*,|v_n|^*)\le \J(u_n,v_n)$, based on the properties:
	\begin{equation*}
	[|u_n|^*]_s \le [|u_n|]_s,~ \int_{\R^N}(I_{\mu}*(|u_n|^*)^p)(|v_n|^*)^q \ge \cnq,
	\end{equation*}
where $|u_n|^*$ and $|v_n|^*$ denotes the symmetric decreasing rearrangement of $|u_n|$ and $|v_n|$ respectively. This allows us to assume without loss of generality that $(u_n,v_n) \in S\D $ is nonnegative and radially decreasing for every $n$. 

By Lemma \ref{a9} and Corollary \ref{a8}, we obtain  $(k_n \star(u_n,v_n))\in \mathcal{P}^{+}\D ,$ and satisfies $([u_n]_s^2+[v_n]_s^2)^{\frac{1}{2}} <R_*$ along with energy minimization: 	
	\begin{align*}
	\J(k_n\star(u_n,v_n))&=\min\{\J(k\star(u_n,v_n)):k\in \R ~and~ ([k_n\star u]_s^2+[k_n\star v]_s^2)^{\frac{1}{2}} <R_* \}\\
	&\le \J(u_n,v_n).
	\end{align*}
This gives rise to a new minimizing sequence, $\{(w_{u_n},w_{v_n})\} \in \mathcal{P}^{+}\D  \cap S_r\D$ and $P_{\alpha}(w_{u_n},w_{v_n})=0$, where $(w_{u_n},w_{v_n})=(k_n\star u_n,k_n\star v_n)$. By Corollary \ref{a10}, we have for all $n$ that $([w_{u_n}]_s^2+[w_{v_n}]_s^2)^{\frac{1}{2}}<R_*-\rho$. Since $(w_{u_n},w_{v_n})\in\mathcal{P}\D$ and $\J(|w_{u_n}|,|w_{v_n}|)=\J(w_{u_n},w_{v_n})$, we can assume that $w_{u_n},w_{v_n}$ are nonnegative. Applying Ekelands's variational principle, there is another minimizing sequence $\{(\tilde{u}_n,\tilde{v}_n)\}$, for which
$m_{\alpha}\D <0$ with $\|(\tilde{u}_n,\tilde{v}_n)-(w_{u_n},w_{v_n})\|_{\h} \to 0$ as $n\to \infty$. This sequence also serves as a Palais-Smale sequence for $\J$ on $S\D$. As a consequence, we obtain
$$\Big([\tilde{u}_n]_s^2+[\tilde{v}_n]_s^2\Big)^{\frac{1}{2}}<R_*-\rho,\quad P_{\alpha}(\tilde{u}_n,\tilde{v}_n)\to 0~\text{as}~n\to\infty.$$	
Subsequently, $\{(\tilde{u}_n,\tilde{v}_n)\}$ satisfies all the conditions of Proposition \ref{a3}. 

To proceed, we assume that alternative \textit{(ii)} from Proposition \ref{a3} does not hold. Then, up to a subsequence, we would have $(\tilde{u}_n,\tilde{v}_n)\rightharpoonup(\tilde{u},\tilde{v})$ weakly in $\h$, but not strongly, with $\tilde{u}\not\equiv0$ and $\tilde{v}\not\equiv0$. In this scenario, $(\tilde{u},\tilde{v})$ represents a solution to \eqref{A_1} for some $\tilde{\lambda}_1,\tilde{\lambda}_2<0$ and 
\begin{equation}\label{A_2}
    c\ge E_{\alpha}(\tilde{u},\tilde{v})+\frac{ab\shl^3}{2}+\frac{b^3\shl^6}{12}+\frac{2}{3}\left(\frac{b^2\shl^4}{4}+a\shl \right)^{\frac{3}{2}} .
\end{equation}
Moreover, the pair satisfies the Pohozaev identity
\begin{align*}
\tilde{P}_{\alpha}(\tilde{u},\tilde{v}):= (a+D_1b)[\tilde{u}]_s^2+(a+D_2b)[\tilde{v}]_s^2-2\int_{\R^N}(I_{\mu}*|\tilde{u}|^{\2})|\tilde{v}|^{\2}-\alpha(\dpq)\int_{\R^N}(I_{\mu}*|\tilde{u}|^{p})|\tilde{v}|^{q}=0.
\end{align*} 
Now, using the fact that $(\tilde{u},\tilde{v})$ solves \eqref{A_1} and satisfies $\|\tilde{u}\|_2^2\le d_1^2$, $\|\tilde{v}\|_2^2\le d_2^2$, we have
\begin{align*}
    E_{\alpha}(\tilde{u},\tilde{v})&=\frac{a}{3}\left([\tilde{u}]_s^2+[\tilde{v}]_s^2\right)+\frac{b}{12}\left(D_1[\tilde{u}]_s^2+D_2[\tilde{v}]_s^2\right) -\frac{\alpha}{6}\left(6-(\dpq)\right)\int_{\R^N}(I_{\mu}*|\tilde{u}|^p)|\tilde{v}|^q\\ 
    &\ge\frac{b}{24}\left([\tilde{u}]_s^2)+[\tilde{v}]_s^2\right)^2-\frac{\alpha}{6}\left(6-(\dpq)\right)B_{N,p,q,s,\mu}(d_1^2+d_2^2)^{\frac{p+q-(\dpq)}{2}}\left([\tilde{u}]_s^2+[\tilde{v}]_s^2\right)^{\frac{\dpq}{2}}\\
&=f\left(([\tilde{u}]_s^2+[\tilde{v}]_s^2)^{\frac{1}{2}}\right)
\end{align*}
where we define $f(t):=\frac{b}{24}t^4-\frac{\alpha}{6}\left(6-(\dpq)\right)B_{N,p,q,s,\mu}(d_1^2+d_2^2)^{\frac{p+q-(\dpq)}{2}}t^{\dpq}$, and minimum of $f$ occurs at $$\tilde{t}:=\left(\frac{\alpha(\dpq)(6-(\dpq))B_{N,p,q,s,\mu}(d_1^2+d_2^2)^{\frac{p+q-(\dpq)}{2}}}{b}\right)^{\frac{1}{4-(\dpq)}}.$$ Thus,
\begin{align}\label{A_3}
   \min_{t\ge0} f(t)=f(\tilde{t})&=-\frac{b}{6}\left(\frac{1}{\dpq}-\frac{1}{4}\right)\tilde{t}^4>-\left(\frac{ab\shl^3}{2}+\frac{b^3\shl^6}{12}+\frac{2}{3}\left(\frac{b^2\shl^4}{4}+a\shl \right)^{\frac{3}{2}}\right)
\end{align} 
as, $\alpha<\alpha_*^*.$ Since $E_{\alpha}(\tilde{u},\tilde{v})\ge f\left(([\tilde{u}]_s^2+[\tilde{v}]_s^2)^{\frac{1}{2}}\right)$, therefore by \eqref{A_2} and \eqref{A_3}, we obtain
$$0>c\le E_{\alpha}(\tilde{u},\tilde{v})+\frac{ab\shl^3}{2}+\frac{b^3\shl^6}{12}+\frac{2}{3}\left(\frac{b^2\shl^4}{4}+a\shl \right)^{\frac{3}{2}}>0,$$
which gives a contradiction, implying our assumption was false. Hence, $(\tilde{u}_n,\tilde{v}_n)\to (\tilde{u},\tilde{v})$ strongly in $\h$. It can be verified that $(\tilde{u},\tilde{v})$ is nonnegative and radially decreasing; by the maximum principle \cite{cabre2014nonlinear}, we have $\tilde{u},\tilde{v}>0$. \par 
The pair $(\tilde{u},\tilde{v})$ is the minimizer, for $m_{\alpha}\D$ on the set $\mathcal{B}_{R_*}$ and since the critical set of $\J|_{S\D}$ lies within $\mathcal{P}\D$, Corollary \ref{a10} ensures $$\J(\tilde{u},\tilde{v})=m_{\alpha}\D=\displaystyle\inf_{\mathcal{P}\D}\J,$$ so $(\tilde{u},\tilde{v})$ is indeed a ground state solution. \par 
To show that any ground state of $\J|_{S\D}$ is a local minimizer of $\J$ in $\mathcal{B}_{R_*}$, take any critical point $(u,v)$ of $\J|_{S\D}$ with $\J(u,v)=m_{\alpha}\D=\displaystyle\inf_{\mathcal{P}\D}\J$. Since $\J(u,v)<0<\displaystyle\inf_{\mathcal{P}^-\D}\J$, we conclude $(u,v)\in \mathcal{P^+}\D$. Again by Corollary \ref{a10}, $\mathcal{P}^+\D\subset\mathcal{B}_{R_*}$, which implies $\big([u]_s^2+[v]_s^2)^{\frac{1}{2}}<R_*$, and therefore we conclude that $(u,v)$ is a local minimizer for $\J|_{\mathcal{B}_{R_*}}$.\par 
Finally, by Lemma \ref{a30}, we have $R_{*}\to 0$ as $\alpha \to 0^+$, and thus $\big([\tilde{u}]_s^2+[\tilde{v}]_s^2)^{\frac{1}{2}}<R_*(d_1,d_2,\alpha)\to0. $ Additionally,
\begin{align*}
    0>m_{\alpha}\D=\J(\tilde{u},\tilde{v})\ge  \frac{a}{2}T+ \frac{b}{8}T^2-\frac{1}{\2(2\shl)^{\2 }} T^{\2}-\alpha B_{N,p,q,s,\mu}(d_1^2+d_2^2)^{\frac{p+q-(\dpq)}{2}}T^{\frac{\dpq}{2}}\to 0,
\end{align*}
where $T=[\tilde{u}]_s^2+[\tilde{v}]^2_s$, indicating $m_{\alpha}\D\to 0$ as $\alpha\to 0^+.$

\section{Proof of Theorem \ref{AA2}}\label{S5}
In this section, we derive the M-P type solution of \eqref{G} for the range $4+\frac{8s-2\mu}{N}<p+q <2\2$. To this end, we utilize the concept of monotonicity, which is precisely established within this section.
\begin{lemma}\label{a11}
	\cite[Lemma 5.1]{li2022normalized}	Let $\tilde{a},\tilde{b},\tilde{c},\tilde{d},\tilde{p},\tilde{q}>0$ and $f(t):=\tilde{a}t^2+\tilde{b}t^4-\tilde{c}t^{\tilde{p}}-\tilde{d}t^{\tilde{q}}$ for $t\ge0$. If $\tilde{p},\tilde{q}\in (4,+\infty)$, $f(t)$ attains a unique maximum point at a positive level on $[0,+\infty)$.
	\end{lemma}
\begin{lemma}\label{a40}
	Let $s\in (0,1),~a,~b,~d_1,~d_2>0,~4+\frac{8s-2\mu}{N}<p+q <2\2$ and $\alpha>0 $. For every $(u,v)\in S\D, ~\f$ has a unique critical point $k_{(u,v)}\in \R$, which is a strict maximum point at a positive level. Moreover, 
	\begin{itemize}
		\item[(i)] $\mathcal{P}\D=\mathcal{P}^{-}\D $.
		\item[(ii)] $\f$ is strictly decreasing on $(k_{(u,v)},+\infty)$, and $k_{(u,v)} <0 \implies \Pa(u,v)<0.$
		\item[(iii)] The map $(u,v) \in S\D \mapsto k_{(u,v)}\in \R$ is of class $\mathcal{C}^1$.
		\item[(iv)] If $\Pa(u,v) <0,$ then $ k_{(u,v)} <0$.
		\end{itemize}   
\end{lemma}	
\begin{proof}
    With the help of Lemma \ref{a11}, we obtain that $\f$ has a unique maximum point at a positive level. The remaining part can be proven using a similar approach as in \cite[Lemma 6.1]{soave2020normalized2}.  
\end{proof}
\begin{lemma}\label{a38}
For $4+\frac{8s-2\mu}{N}<p+q <2\2$, we have 
$$m_{\alpha}\D:=\inf_{\mathcal{P}\D}\J(u,v)>0.$$  	
\end{lemma}
\begin{proof}
    Let an arbitrary element $(u,v)\in\mathcal{P}\D$. Using $\Pa(u,v)=0$, we obtain
    \begin{align*}
        a([u]_s^2+[v]_s^2)+\frac{b}{2}([u]_s^2+[v]_s^2)^2&<\alpha(\dpq)\cpq+2\cq\\
        &<\alpha B_{N,p,q,s,\mu}(d_1^2+d_2^2)^{\frac{p+q-(\dpq)}{2}}(\dpq)([u]_s^2+[v]_s^2)^{\frac{\dpq}{2}}+2\left(\frac{[u]_s^2+[v]_s^2}{2\shl}\right)^{\2},
    \end{align*}
    which implies $\displaystyle\inf_{(u,v)\in\mathcal{P}\D}([u]_s^2+[v]^2_s)\ge C_0>0$, for $\dpq>4$, and hence by definition of $\mathcal{P}\D$, we have
    \begin{equation*}
        \inf_{(u,v)\in\mathcal{P}\D}\left(\int_{\R^N}(I_{\mu}*|u|^{p})|v|^q+\int_{\R^N}(I_{\mu}*|u|^{\2})|v|^{\2}\right)>0.
    \end{equation*}
    Finally, combining these facts and using the definition of $\J(u,v)$, we conclude 
    \begin{equation*}\aligned
\inf_{(u,v)\in\mathcal{P}\D}\J(u,v) 
= \inf_{(u,v)\in\mathcal{P}\D}\left(
\frac{a}{4}\left([u]_s^2 + [v]_s^2\right)
+ \alpha\left(\frac{\dpq}{4} - 1\right) \int_{\mathbb{R}^N} (I_{\mu} * |u|^p) |v|^q \right. \\
\quad \left. + \left(\frac{1}{2} - \frac{1}{\2}\right) \int_{\mathbb{R}^N} (I_{\mu} * |u|^{\2}) |v|^{\2} 
\right) > 0,\endaligned
\end{equation*}
as desired.
\end{proof}
\begin{lemma}\label{c56}
	There exists $z>0$, sufficiently small such that 
$$0<\sup_{\overline{\mathcal{B}}_z}\J<m_{\alpha}\D$$	and	$$(u,v)\in {\overline{\mathcal{B}}_z} \implies \J(u,v) >0,~\Pa(u,v)>0,$$
	where $\mathcal{B}_z:=\{(u,v)\in S\D :([u]_s^2+[v]_s^2)^{\frac{1}{2}}<z \}$.
\end{lemma}
\begin{proof}
    Applying \eqref{Ai2} and \eqref{Ai4}, we have
    \begin{align*}
        \J(u,v)\ge&\frac{a}{2}\Big([u]_s^2+[v]_s^2  \Big)+ \frac{b}{8}\Big([u]_s^2+[v]_s^2  \Big)^2-\frac{1}{\2} \Big(\frac{[u]_s^2+[v]_s^2}{2\shl}  \Big)^{\2}
-\\
&\alpha B_{N,p,q,s,\mu}(d_1^2+d_2^2)^{\frac{p+q-(\dpq)}{2}}\Big([u]_s^2+[v]_s^2  \Big)^{\frac{\dpq}{2}}>0,
    \end{align*} 
    and
    \begin{align*}
        P_{\alpha}(u,v)\ge &a\Big([u]_s^2+[v]_s^2  \Big)+ \frac{b}{2}\Big([u]_s^2+[v]_s^2  \Big)^2-{2}\Big(\frac{[u]_s^2+[v]_s^2}{2\shl}  \Big)^{\2}
\\&-\alpha B_{N,p,q,s,\mu}(d_1^2+d_2^2)^{\frac{p+q-(\dpq)}{2}}(\dpq)\Big([u]_s^2+[v]_s^2  \Big)^{\frac{\dpq}{2}}>0,
    \end{align*} 
    if $(u,v)\in {\overline{\mathcal{B}}_z}$ with $z$ small enough, since $\dpq>4.$ If necessary replacing $z$ with a small quantity, we also have $\J(u,v)\le  \frac{a}{2}\Big([u]_s^2+[v]_s^2  \Big)+ \frac{b}{4}\Big([u]_s^4+[v]_s^4  \Big)<m_{\alpha}\D$, for every $(u,v)\in \overline{\mathcal{B}}_z$.
\end{proof}
\begin{lemma}\label{a12}
	Let $d_1,~d_2,~\tilde{d_1},~\tilde{d_2}>0$ such that $\tilde{d_1}<d_1$ and $\tilde{d_2}<d_2$. Then $m_{\alpha}\D <m_{\alpha}(\tilde{d_1},\tilde{d_2})$.
\end{lemma}
\begin{proof}
	To prove this lemma, it is sufficient to show that for arbitrary $\epsilon>0$, one has 
	\begin{equation}\label{A21}
	m_{\alpha}\D \le m_{\alpha}(\dg,\df)+\epsilon.
	\end{equation}   
	By definition of $m_{\alpha}(\dg,\df)$, there exists $(u,v)\in \mathcal{P}(\dg,\df)$, such that 
\begin{equation}\label{A22}
	\J(u,v)\le m_{\alpha}(\dg,\df,)+\frac{\epsilon}{2}.
\end{equation}
	Let $\phi \in C_0^{\infty}(\R^N)$ be radial function, such that $0\le\phi\le1$ and 
	\begin{equation*}\phi(x)=
	\begin{cases}
	1 &\text{if} ~\|x\|_2 \le 1,\\
	0 &\text{if}~\|x\|_2 \ge 2 .
	\end{cases}
	\end{equation*}
	For $\delta>0$, we define
	$u_{\delta}(x):=u(x)\phi(\delta x)~and~v_{\delta}(x):=v(x)\phi(\delta x)$,
	so that $(u_{\delta},v_{\delta}) \to (u,v)$ in $\h$ as $\delta \to 0^+$.
	Then by Corollary \ref{a0}, $t_{\alpha}\star(u_{\delta},v_{\delta})\to t_{\alpha}\star(u,v)$ in $\h$ as $\delta \to 0^+$. Fix $\delta>0$, such that 
	\begin{equation}\label{A23}
\J(t_{\alpha}\star(u_{\delta},v_{\delta}))\le \J(u,v)+\frac{\epsilon}{4}.	
	\end{equation}

	Next, we choose $\psi \in C_0^{\infty}(\R^N)$, with $supp(\psi)
\subset \R^N\backslash B\left(0,\frac{4}{\delta}\right)$ and set 
\begin{equation*}\psi_{d_1}=\dfrac{\sqrt{d_1^2-\|u\|_2^2}}{\|\psi\|_2^2}\psi\quad \text{and}\quad  \psi_{d_2}=\dfrac{\sqrt{d_2^2-\|v\|_2^2}}{\|\psi\|_2^2}\psi,
\end{equation*}
and, thus
$$(\tilde{u}_{\tau},\tilde{v}_{\tau}):=(u_{\delta}+\tau\star\psi_{d_1},v_{\delta}+\tau\star\psi_{d_2})\in S\D,$$
as $\left(supp(u_{\delta})\cup supp(v_{\delta})\right)\cap \left(supp(\tau\star\psi_{d_1})\cup supp(\tau\star\psi_{d_2})\right)=\emptyset
$ for $\tau\le 0$.\\
By Lemma \ref{a40}, there exists a unique $t_{\tau}$, such that $\Pa(t_{\tau}\star(\tilde{u}_{\tau},\tilde{v}_{\tau}))=0$, as a result, we have
\begin{align*}
\frac{a}{e^{(2\2-2)st_{\tau}}}([\tilde{u}_{\tau}]_s^2+[\tilde{v}_{\tau}]_s^2)+\frac{b}{e^{(2\2-4)st_{\tau}}}([\tilde{u}_{\tau}]^4_s+[\tilde{v}_{\tau}]^4_s)&=
\frac{\alpha(\dpq)}{e^{(2\2-(\dpq))st_{\tau}}}\int_{\R^N} (I_{\alpha}*|\tilde{u}_{\tau}|^{p})|\tilde{v}_{\tau}|^{q}\\
&
+2\int_{\R^N} (I_{\alpha}*|\tilde{u}_{\tau}|^{\2})|\tilde{v}_{\tau}|^{\2}.
\end{align*}
It follows that $\limsup\limits_{\tau \to -\infty}t_{\tau}<+\infty$, because $(\tilde{u}_{\tau},\tilde{v}_{\tau}) \to (u_{\delta},v_{\delta})\not=0 $ as $\tau\to -\infty$.
\par 
Consequently $t_{\tau}+\tau \to -\infty$ as $\tau\to -\infty$ and  for small $\tau$ enough, we have 
\begin{align}\label{A24}
2\left( \int_{\R^N}\int_{\R^N}
\dfrac{(t_{\tau}\star u_{\delta}(x))((t_{\tau}+\tau)\star\psi_{d_1}(x))}{|x-y|^{N+2s}}\dx\dy+\int_{\R^N}\int_{\R^N}\dfrac{(t_{\tau}\star v_{\delta}(x))((t_{\tau}+\tau)\star\psi_{d_2}(x))}{|x-y|^{N+2s}}\dx\dy\right) <\frac{\epsilon}{8}
\end{align}

and 
\begin{equation}\label{A25}
\J\big((t_{\tau}+\tau)\star(\psi_{d_1},\psi_{d_2})\big)<\frac{\epsilon}{8}.
\end{equation}
Now using equation \eqref{A22}-\eqref{A25} in \eqref{A21}, we deduce
\begin{align*}
m_{\alpha}\D &\le \J(t_{\tau}\star(\tilde{u}_{\tau},\tilde{v}_{\tau}))=\J\left(t_{\tau}\star(u_{\delta},v_{\delta})+(t_\tau+\tau)\star(\psi_{d_1},\psi_{d_2})\right)\\
& \le \J(t_{\tau}\star(u_{\delta},v_{\delta}))+\J((t_{\tau}+\tau)\star(\psi_{d_1},\psi_{d_2}))+\frac{\epsilon}{8}\\
&\le \J(t_{\tau}\star(u_{\delta},v_{\delta}))+\frac{\epsilon}{4}\\
&\le m_{\alpha}(\dg,\df)+\epsilon,
\end{align*} 
which completes the proof.
\end{proof}

\begin{lemma}\label{lemma_5.6}	Let $s\in (0,1),~a,~d_1,~d_2>0,~4+\frac{8s-2\mu}{N}<p+q <2\2=6$ and $\alpha>0$. Then $$m_{r,\alpha}\D<\frac{ab\shl^3}{2}+\frac{b^3\shl^6}{12}+\frac{2}{3}\left(\frac{b^2\shl^4}{4}+a\shl \right)^{\frac{3}{2}}. $$
	
\end{lemma}
	\begin{proof} Let $u_{\epsilon}:=d_1\frac{ w_{\epsilon}}{\|w_{\epsilon}\|_2}$, where $w_{\epsilon}$ is define in \eqref{a102}. One can observe that $(u_{\epsilon},u_{\epsilon})\in S(d_1,d_1)$, by Lemma \ref{a40}, there exists a unique $t_{\epsilon,\alpha}\in \R$, such that 
    \begin{equation*}
        m_{r,\alpha}(d_1,d_1)\le \J(t_{\epsilon,\alpha}\star(u_{\epsilon},u_{\epsilon}))=\displaystyle\max_{k\in\R}\J(k\star (u_{\epsilon},u_{\epsilon}))=\max_{k\in\R}\f(k)\quad \forall \epsilon>0.
    \end{equation*}
%By the help of Proposition $2.7$ and Proposition $2.8$ of \cite{giacomoni2018doubly}, we get the following integral estimates 
%\begin{equation*}
%[\we]_s^2= \shl^{\frac{2N-\mu}{N-\mu+2s}}+O(\epsilon^{N-2s}),
%\end{equation*}
%\begin{align*}
%\int_{\R^N}(I_{\mu}*|\we|^{\2})|\we|^{\2}\le
%\shl^{\frac{2N-\mu}{N-\mu+2s}}+O(\epsilon^{N}),\\
%\int_{\R^N}(I_{\mu}*|\we|^{\2})|\we|^{\2}\ge \shl^{\frac{2N-\mu}{N-\mu+2s}}-O(\epsilon^N),
%\end{align*}
%\begin{equation*}
%\int_{\R^N}|\we|^2=\begin{cases}
%\epsilon^{2s}+O(\epsilon^{N-2s})&N>4s,\\
%\epsilon^{2s}|\log\epsilon|+O(\epsilon^{2s})&N=4s,%\\
%\epsilon^{N-2s}+O(\epsilon^{2s})&N<4s.\\
%\end{cases}
%\end{equation*}
Without loss of generality we can assume that $d_1<d_2$, then by Lemma \ref{a12}, we have 
 \begin{equation*}
     m_{r,\alpha}\D<m_{r,\alpha}(d_1,d_1).
 \end{equation*}
So, it is sufficient to show that $$\displaystyle\max_{k\in \R}\f(k)=\J(t_{\epsilon,\alpha}\star (u_{\epsilon},u_{\epsilon}))<\frac{ab\shl^3}{2}+\frac{b^3\shl^6}{12}+\frac{2}{3}\left(\frac{b^2\shl^4}{4}+a\shl \right)^{\frac{3}{2}} .$$
First, we prove 
$$\displaystyle\max_{k\in \R}\Phi_0(k)=\mathcal{J}_0(t_{\epsilon,0}\star(u_{\epsilon},
u_{\epsilon})) =\frac{ab\shl^3}{2}+\frac{b^3\shl^6}{12}+\frac{2}{3}\left(\frac{b^2\shl^4}{4}+a\shl \right)^{\frac{3}{2}} +O(\epsilon^{\frac{N-2s}{2}}).$$ Since 
\begin{align*}
\Phi_0(t)
=ae^{2st}\frac{d_1^2}{\|w_{\epsilon}\|_2^2}[w_{\epsilon}]_s^2+\frac{be^{4st}d_1^4}{2\|w_{\epsilon}\|_2^4}[w_{\epsilon}]_s^4-\frac{e^{6st}d_1^6}{3\|w_{\epsilon}\|_2^6}\int_{\R^N}\left(I_{\mu}*|w_{\epsilon}|^3\right)|w_{\epsilon}|^3,  
\end{align*}

we deduce $\Phi_0(t)$ has a maximum point $t_{\epsilon,0}$, such that 
\begin{align*}
\frac{d_1^2e^{2s\teo}}{\|w_{\epsilon}\|^2_2}&=\frac{b[\we]_s^4}{2\int_{\R^N}(I_{\mu}*|\we|^3)|\we|^3}+\sqrt{\frac{b^2[\we]_s^8}{4((\int_{\R^N}(I_{\mu}*|\we|^3)|\we|^3)^2}+\frac{a[\we]_s^2}{\int_{\R^N}(I_{\mu}*|\we|^3)|\we|^3}}\\
&\le \frac{b(\shl^{\frac{3}{2}}+O(\epsilon^{N-2s}))^2}{2(\shl^{\frac{3}{2}}-O(\epsilon^N))}+\sqrt{\frac{b^2(\shl^{\frac{3}{2}}+O(\epsilon^{N-2s}))^4}{4(\shl^{\frac{3}{2}}-O(\epsilon^N))^2}+\frac{a(\shl^{\frac{3}{2}}+O(\epsilon^{N-2s}))}{(\shl^{\frac{3}{2}}-O(\epsilon^N))} }\\
&=\frac{b\shl^{\frac{3}{2}}}{2}+\sqrt{a+\frac{b^2\shl^3}{4}+O(\epsilon^{N-2s})}+O(\epsilon^{N-2s})\\
&\le \frac{b\shl^{\frac{3}{2}}}{2}+\sqrt{a+\frac{b^2\shl^3}{4}}+O(\epsilon^{\frac{N-2s}{2}})
=\frac{\zeta}{\sqrt{\shl}}+O(\epsilon^{\frac{N-2s}{2}}),
\end{align*} 
where $\zeta=\dfrac{b\shl^2}{2}+\sqrt{\dfrac{b^2\shl^4}{4}+a\shl}$.
Similarly, using \eqref{Ai}, we have
\begin{align*}
\frac{d_1^2e^{2s\teo}}{\|w_{\epsilon}\|^2_2}\ge\frac{\zeta}{\sqrt{\shl}}.
\end{align*} 
 Which leads to 
\begin{align*}
\sup_{k\in\R}\Phi_{0}(k)&=\Phi_{0}(\teo) =a\frac{d_1^2e^{2s\teo}}{\|w_{\epsilon}\|^2_2}[w_{\epsilon}]_s^2+\frac{be^{4st}}{2}\frac{d_1^4e^{4s\teo}}{\|w_{\epsilon}\|^4_2}[w_{\epsilon}]_s^4-\frac{d_1^6e^{6s\teo}}{3\|w_{\epsilon}\|^6_2}\int_{\R^N}(I_{\mu}*|w_{\epsilon}|^3)|w_{\epsilon}|^3\\
\notag &\le a\left(\frac{\zeta}{\sqrt{\shl}}+O(\epsilon^{\frac{N-2s}{2}}) \right)\left(\shl^{\frac{3}{2}}+O(\epsilon^{N-2s}) \right)\\ \notag &+
\frac{b}{2}\left(\frac{\zeta}{\sqrt{\shl}}+O(\epsilon^{\frac{N-2s}{2}}) \right)^2\left(\shl^{\frac{3}{2}}+O(\epsilon^{N-2s}) \right)^2-\left(\frac{\zeta}{\sqrt{\shl}}\right)^3\frac{\shl^{\frac{3}{2}}-O(\epsilon^N)}{3}\\
&=a\zeta\shl+\frac{b\zeta^2\shl^2}{2}-\frac{\zeta^3}{3}+O(\epsilon^{\frac{N-2s}{2}}) \\
&<\frac{2a\shl\zeta}{3}+\frac{b\shl^2\zeta^2}{6}+O(\epsilon^{\frac{N-2s}{2}})\\
&=\frac{ab\shl^3}{2}+\frac{b^3\shl^6}{12}+\frac{2}{3}\left(\frac{b^2\shl^4}{4}+a\shl \right)^{\frac{3}{2}} +O(\epsilon^{\frac{N-2s}{2}}).
\end{align*}
To get an estimate for $t_{\epsilon,\alpha}$, from $(\f)'(t_{\epsilon,\alpha})=\Pa(t_{\epsilon,\alpha}\star (u_{\epsilon},u_{\epsilon}))=0$, we obtain
 \begin{align*}
  ae^{2s\twe}d_1^2\frac{[\we]_s^2}{\|\we\|_2^2}+ be^{4s\twe}d_1^4\frac{[\we]_s^4}{\|\we\|_2^4}&=\alpha(\dpq)e^{(\dpq)s\twe}\frac{d_1^{p+q}}{\|\we\|^{p+q}}\int_{\R^N}(I_{\mu}*|\we|^p)|\we|^q \\
  &+\frac{e^{6s\twe}d_1^6}{\|\we\|^6_2}\int_{\R^N}(I_{\mu}*|\we|^3)|\we|^3
  \\&>\frac{e^{6s\twe}d_1^6}{\|\we\|^6_2}\int_{\R^N}(I_{\mu}*|\we|^3)|\we|^3.
 \end{align*}
 It results in that $\frac{e^{2s\twe}d_1^2}{\|\we\|_2^2}<\frac{e^{2st_{\epsilon,0}}d_1^2}{\|\we\|_2^2}$, so we have
 \begin{align}\label{A30}
 \notag \frac{e^{2s\twe}d_1^2}{\|\we\|_2^2}&
\le \frac{b[\we]_s^4}{\int_{\R^N}(I_{\mu}*|\we|^3)|\we|^3}+\sqrt{\frac{b^2[\we]_s^8}{4\int_{\R^N}(I_{\mu}*|\we|^3)|\we|^3)^2} +\frac{a[\we]_s^2}{\int_{\R^N}(I_{\mu}*|\we|^3)|\we|^3}}
\\&\le \frac{b[\we]_s^4}{\int_{\R^N}(I_{\mu}*|\we|^3)|\we|^3}+\frac{\sqrt{a}[\we]_s}{\sqrt{\int_{\R^N}(I_{\mu}*|\we|^3)|\we|^3}}.
 \end{align}
 On the other hand, we deduce 
 \begin{align*}
  \frac{e^{4s\twe}d_1^4}{\|\we\|_2^4}&=\frac{a[\we]_s^2}{\int_{\R^N}(I_{\mu}*|\we|^3)|\we|^3} +\frac{be^{2s\twe}d_1^2[\we]_s^4}{\|\we\|_2^2\int_{\R^N}(I_{\mu}*|\we|^3)|\we|^3}\\
  &-\alpha(\dpq)e^{(\dpq-2)s\twe}\frac{d_1^{p+q-2}\int_{\R^N}(I_{\mu}*|\we|^p)|\we|^q}{\|\we\|_2^{p+q-2}\int_{\R^N}(I_{\mu}*|\we|^3)|\we|^3}
  \\&\ge \frac{be^{2s\twe}d_1^2[\we]_s^4}{\|\we\|_2^2\int_{\R^N}(I_{\mu}*|\we|^3)|\we|^3}-\alpha(\dpq)e^{(\dpq-2)s\twe}\frac{d_1^{p+q-2}\int_{\R^N}(I_{\mu}*|\we|^p)|\we|^q}{\|\we\|_2^{p+q-2}\int_{\R^N}(I_{\mu}*|\we|^3)|\we|^3}\\
 \frac{e^{2s\twe}d_1^2}{\|\we\|_2^2}&\ge\frac{b[\we]_s^4}{\int_{\R^N}(I_{\mu}*|\we|^3)|\we|^3}-\alpha(\dpq)\frac{e^{(\dpq-4)s\twe}d_1^{p+q-4}\int_{\R^N}(I_{\mu}*|\we|^p)|\we|^q}{\|\we\|_2^{p+q-4}\int_{\R^N}(I_{\mu}*|\we|^3)|\we|^3}.
 \end{align*}
 
 By the inequality $(l_1+l_2)^{\theta}\le l_1^{\theta}+l_2^{\theta}$ for $l_1,~l_2>0$, $\theta\in (0,1]$ and \eqref{A30}, we deduce
\begin{align}\label{A211}
\notag\frac{e^{2s\twe}d_1^2}{\|\we\|_2^2}&\ge\frac{b[\we]_s^4}{\int_{\R^N}(I_{\mu}*|\we|^3)|\we|^3}-\alpha(\dpq)\frac{d_1^{p+q-4}\int_{\R^N}(I_{\mu}*|\we|^p)|\we|^q}{\|\we\|_2^{p+q-4}\int_{\R^N}(I_{\mu}*|\we|^3)|\we|^3}e^{2st_{\epsilon,\alpha}{\frac{(\dpq-4)}{2}}}\\
\notag&\ge\frac{b[\we]_s^4}{\int_{\R^N}(I_{\mu}*|\we|^3)|\we|^3}-\alpha(\dpq)\frac{d_1^{p+q-4}\int_{\R^N}(I_{\mu}*|\we|^p)|\we|^q}{\|\we\|_2^{p+q-4}\int_{\R^N}(I_{\mu}*|\we|^3)|\we|^3}\\&\notag\quad\times\Bigg(\frac{\|\we\|_2^2}{d_1^2}\Bigg( \frac{b[\we]_s^4}{\int_{\R^N}(I_{\mu}*|\we|^3)|\we|^3}\quad +\frac{\sqrt{a}[\we]_s}{\sqrt{\int_{\R^N}(I_{\mu}*|\we|^3)|\we|^3}} \Bigg)\Bigg)^{\frac{\dpq-4}{2}}\\
&\notag\ge\frac{b[\we]_s^4}{\int_{\R^N}(I_{\mu}*|\we|^3)|\we|^3}-\alpha(\dpq)\frac{d_1^{p+q-(\dpq)}\int_{\R^N}(I_{\mu}*|\we|^p)|\we|^q}{\|\we\|_2^{p+q-(\dpq)}\int_{\R^N}(I_{\mu}*|\we|^3)|\we|^3}\\&\quad\times\Bigg( \frac{b[\we]_s^4}{\int_{\R^N}(I_{\mu}*|\we|^3)|\we|^3}+\frac{\sqrt{a}[\we]_s}{\sqrt{\int_{\R^N}(I_{\mu}*|\we|^3)|\we|^3}}\Bigg)^{\frac{\dpq-4}{2}}.
\end{align}  
Therefore, we obtain
\begin{align*}  
e^{2s\twe}&\ge\frac{\|\we\|_2^2}{d_1^2}\left(K_1-\alpha K_2(\dpq)d_1^{p+q-(\dpq)}\frac{\int_{\R^N}(I_{\mu}*|\we|^p)|\we|^q }{\|\we\|_2^{p+q-(\dpq)}} \right)\\
&=\frac{\|\we\|_2^2}{d_1^2}\left(K_1-\alpha K_2(\dpq)d_1^{p+q-(\dpq)}f(\epsilon) \right)
\end{align*}
  
  where $K_1=K_1(b,\shl)>0$, $K_2=K_2(a,b,p,q,\shl)>0$ and with help of Lemma \ref{Ak3}, we define $f(\epsilon)$ as
  
  \begin{equation}\label{ar1}
  f(\epsilon):=\frac{\int_{\R^N}(I_{\mu}*|\we|^p)|\we|^q }{\|\we\|_2^{p+q-(\dpq)}}=    \begin{cases}
\textbf{Case 1: } 4s<N<6s, \\[4pt]
\quad
\begin{cases}
 \text{Constant}~~, & p>\tfrac{3}{2},\; q>\tfrac{3}{2},
\end{cases} \\[12pt]

\textbf{Case 2: } N=4s, \\[4pt]
\quad
\begin{cases}
|\log \epsilon|^{\tfrac{p+q-6}{2}}, & p>\tfrac{3}{2},\; q>\tfrac{3}{2}, \\[4pt]

|\log \epsilon|^{\tfrac{-3+\max\{p,q\}}{2}}, & \big(p>\tfrac{3}{2},\,q=\tfrac{3}{2}\big)\;\text{or}\;\big(p=\tfrac{3}{2},\,q>\tfrac{3}{2}\big),
\end{cases}\\[12pt]
\textbf{Case 3: } 2s<N<4s, \\[4pt]
\quad
\begin{cases}
\epsilon^{\tfrac{(N-2s)(4s-N)(6-(p+q))}{4s}}, & p>\tfrac{3}{2},\; q>\tfrac{3}{2}, \\[4pt]
\epsilon^{\tfrac{(N-2s)}{4s}\big(18s+N(p+q)-4s\max\{p,q\}-6N\big)}, & \big(p>\tfrac{3}{2},\,q<\tfrac{3}{2}\big)\;\text{or}\;\big(p<\tfrac{3}{2},\,q>\tfrac{3}{2}\big), \\[4pt]
\epsilon^{\tfrac{(N-2s)(9-2\max\{p,q\})(4s-N)}{8s}}\,|\log \epsilon|^{\frac{3(N-2s)}{2N}}, & \big(p>\tfrac{3}{2},\,q=\tfrac{3}{2}\big)\;\text{or}\;\big(p=\tfrac{3}{2},\,q>\tfrac{3}{2}\big).
\end{cases}
\end{cases}
  \end{equation}

%  \begin{equation}\label{ar1}
 % \frac{\int_{\R^N}(I_{\mu}*|\we|^p)|\we|^q }{\|\we\|_2^{p+q-(\dpq)}}=\begin{cases}
 % O(|\log\epsilon|^{\frac{2p-5}{4}})~~~~~~~~~~N=4s,~p>\frac{5}{2},~q=\frac{3}{2},\\
%  O(|\log\epsilon|^{\frac{2q-5}{4}})~~~~~~~~~~N=4s,~q>\frac{5}{2},~p=\frac{3}{2},\\
  %O(|\log\epsilon|)~~~~~~~~~~~~N>4s,~p>\frac{3}{2},~q=\frac{3}{2},\\
%O(|\log\epsilon|)~~~~~~~~~~~~N>4s,~q>\frac{3}{2},~p=\frac{3}{2},\\
 %  O(\epsilon^{(\frac{N-2s}{4s})(6-(p+q))(4s-N)})~\frac{10s}{3}<N<4s,~q>\frac{3}{2},~p>\frac{3}{2},\\
  % O(\epsilon^{(\frac{N-2s}{2s})(18s+Np-6N+Nq-4ps)})~~2s<N<3s,~p>\frac{3}{2},~q<\frac{3}{2},\\
   % O(\epsilon^{(\frac{N-2s}{2s})(18s+Nq-6N+Np-4qs)})~~2s<N<3s,~q>\frac{3}{2},~p<\frac{3}{2}.
   %\end{cases}
 % \end{equation}
   From \eqref{ar1}, together with the assumptions \textit{(i)-(v)} in Theorem \ref{AA2}  on $p,~q,~N$ and $\alpha$, we deduce
   \begin{equation*}
   e^{2s\twe}\ge \frac{\|\we\|_2^2K_1}{d_1^2},
   \end{equation*}
for sufficiently small $\epsilon>0$.\par 
Since $4+\frac{8s-2\mu}{N}<p+q<6$, we get
\begin{align}\label{A205}
\notag\sup_{k\in\R}\Phi_{\alpha}(k)&=\Phi_{\alpha}(\twe)=\Phi_0(\twe)-\alpha e^{s(\dpq)\twe}\frac{d_1^{p+q}}{\|\we\|_2^{p+q}}\int_{\R^N}(I_{\mu}*|\we|^p)|\we|^q\\
\notag&\le \sup_{k\in\R}\Phi_{0}(k)-\alpha e^{s(\dpq)\twe}\frac{d_1^{p+q}}{\|\we\|_2^{p+q}}\int_{\R^N}(I_{\mu}*|\we|^p)|\we|^q\\
\notag&= \Phi_{0}(t_{\epsilon,0})-\alpha e^{s(\dpq)\twe}\frac{d_1^{p+q}}{\|\we\|_2^{p+q}}\int_{\R^N}(I_{\mu}*|\we|^p)|\we|^q\\
\notag&\le \frac{2a\shl^3\zeta}{2}+\frac{b^3\shl^6\zeta^2}{12}+O(\epsilon^{\frac{N-2s}{2}})-\alpha\frac{K_1^{\frac{(\dpq)}{2}}d_1^{p+q-(\dpq)}\int_{\R^N}(I_{\mu}*|\we|^p)|\we|^q}{\|\we\|_2^{p+q-(\dpq)}}\\
&=\frac{2a\shl^3\zeta}{2}+\frac{b^3\shl^6\zeta^2}{12}+O(\epsilon^{\frac{N-2s}{2}})-\alpha K_1^{\frac{(\dpq)}{2}}d_1^{p+q-(\dpq)} f(\epsilon).
\end{align}
Again, using assumptions $(i)-(v)$ in Theorem \ref{AA2}, for sufficiently same $\epsilon$, we have
$$O(\epsilon^{\frac{N-2s}{2}})-\alpha K_1^{\frac{(\dpq)}{2}}d_1^{p+q-(\dpq)} f(\epsilon)<0.$$
Therefore, we conclude
\begin{align*}
\sup_{k\in\R}\Phi_{\alpha}(k)&<\frac{2a\shl^3\zeta}{2}+\frac{b^3\shl^6\zeta^2}{12}\\
&=\frac{ab\shl^3}{2}+\frac{b^3\shl^6}{12}+\frac{2}{3}\left(\frac{b^2\shl^4}{4}+a\shl \right)^{\frac{3}{2}},
\end{align*}
we get the desired result.
\end{proof} 
	\textbf{Proof of Theorem \ref{AA2}} Let  $\J^0$ represent the closed sub-level set $\{(u,v)\in S\D:\J\le 0 \}$. Define 
	\begin{equation*}
	\Gamma :=\{\gamma=(\vartheta,\beta)\in C([0,1],\R\times S_r\D ):\gamma(0)=(0,\overline{\mathcal{B}}_z),\gamma(1)\in (0,\J^0)\},
	\end{equation*}
	 with the associated minimax level
	$$\sigma_{\alpha}\D:=\inf_{\gamma\in \Gamma}\max_{(k,u,v)\in \gamma([0,1])}\J(k\star(u,v)).$$
	 Let $(u,v)\in S_r\D$. Since $[k\star u]_s^2+[k\star v]_s^2\to 0^+$ as $k\to -\infty$ and $\J(k\star(u,v))\to -\infty$ as $k\to +\infty$, there exists $k_0<<-1$ and $k_1>>1$, such that 
	$$\gamma_{(u,v)}:\tau \in [0,1] \mapsto (0,((1-\tau)k_0+\tau k_1)\star(u,v))\in \R\times S_r\D$$ 
is a path in $\Gamma$. Then $\sigma_{\alpha}\D$ is a real number. Now, for any $\gamma=(\vartheta,\beta)\in \Gamma$, we introduce a function 
$$K_{\gamma}:k\in [0,1]\mapsto \Pa(\vartheta(k)\star\beta(k))\in \R.$$

By Lemma \ref{a38} and Lemma \ref{c56}, we find that $K_{\gamma}(0)=\Pa(\beta(0))>0$. Moreover, since $\f(k)>0$ for each $k\in (-\infty,k_{\beta(1)})$
 and $\f(k_{\beta(1)})=\J(\beta(1))\le 0$, it follows that $k_{\beta(1)}<0.$
  Therefore, by Lemma \ref{a40}, we have $K_{\gamma}=\Pa(\beta(1))<0.$ Moreover, the map $\tau\mapsto \vartheta(\tau) \star \beta(\tau)$ is continuous from $[0,1]$ to $\h$, so we infer that there exists $\tau_{\gamma}\in (0,1)$ such that $K_{\gamma}(\tau_{\gamma})=0$. Consequently, $\vartheta(\tau_{\gamma})\star\beta(\tau_{\gamma})\in \mathcal{P}$, which implies that
  $$\max_{\gamma([0,1])}\J(k\star(u,v))\ge \J(\vartheta(\tau_{\gamma})\star\beta(\tau_{\gamma})\ge\sigma_{\alpha}\D\ge \inf_{\mathcal{P}\cap S_r\D}\J =m_{r,\alpha}\D.$$
  Therefore, $\sigma_{\alpha}\D\ge m_{r,\alpha}\D.$ On the other hand, if $(u,v)\in \mathcal{P}\cap S_{r}\D$, then $\gamma_{(u,v)}$ is a path in $\Gamma$ with 
  $$\J(u,v) =\max_{\gamma([0,1])}\J(k\star(u,v))\ge \sigma_{\alpha}\D,$$
which implies that
$$m_{r,\alpha}\ge\sigma_{\alpha}\D.$$
Combining this with Lemma \ref{a38}, we have
$$\sigma_{\alpha}\D=m_{r,\alpha}\D>\sup_{(\overline{\mathcal{B}}_z\cup\J^m)\cap S_r\D}\J=\sup_{((0,\overline{\mathcal{B}}_z)\cup(0,\J^m))\cap(\R\times S_r\D)}\J(k\star(u,v)).$$

Following the same approach as in \cite[Theorem 1.1(2)]{soave2020normalized2}, we obtain a Palais-Smale sequence $\{(u_n,v_n)\}\subset S_r\D$ for $\J|_{S\D}$, at a level $m_{\alpha,r}\D\in\left(0,\frac{ab\shl^3}{2}+\frac{b^3\shl^6}{12}+\frac{2}{3}\left(\frac{b^2\shl^4}{4}+a\shl \right)^{\frac{3}{2}}\right)$. This sequence has the properties that $P_{\alpha}(u_n,v_n)\to 0$, $u_n^-\to0$ and $v_n^-\to0$ a.e. in $\R^N$.

At this point, one of the two alternatives of Proposition \ref{a3} occurs. Suppose that alternative \textit{(i)} holds: there exists a nontrivial function $(\tilde{u},\tilde{v})\in\h$ such that $(u_n,v_n)\rightharpoonup(\tilde{u},\tilde{v})$ weakly (but not strongly) and $(\tilde{u},\tilde{v})$ is a solution of $(F_\alpha)$, and satisfies
\begin{equation*}
   \J(\tilde{u},\tilde{v})\le m_{r,\alpha}\D-\left(\frac{ab\shl^3}{2}+\frac{b^3\shl^6}{12}+\frac{2}{3}\left(\frac{b^2\shl^4}{4}+a\shl \right)^{\frac{3}{2}} \right)<0.\end{equation*}
But, since $P_{\alpha}(\tilde{u},\tilde{v})=0$ and $\dpq>4$, implies
\begin{align*}
    \J(\tilde{u},\tilde{v})=
\frac{a}{4}\left([\tilde{u}]_s^2 + [\tilde{v}]_s^2\right)
+ \alpha\left(\frac{\dpq}{4} - 1\right) \int_{\mathbb{R}^N} (I_{\mu} * |\tilde{u}|^p) |\tilde{v}|^q + \left(\frac{1}{2} - \frac{1}{\2}\right) \int_{\mathbb{R}^N} (I_{\mu} * |\tilde{u}|^{\2}) |\tilde{v}|^{\2} 
>0,
\end{align*}
a contradiction. It's evident that alternative \textit{(ii)} of Proposition \ref{a3} is clearly valid, leading us to our desired result.\qed

\printbibliography

@article{guo2017critical,
  title={On critical systems involving fractional Laplacian},
  author={Guo, Zhenyu and Luo, Senping and Zou, Wenming},
  journal={Journal of Mathematical Analysis and Applications},
  volume={446},
  number={1},
  pages={681--706},
  year={2017},
  publisher={Elsevier}
}

@article{li2022normalized,
  title={Normalized solutions to a class of {K}irchhoff equations with {S}obolev critical exponent},
  author={Li, Gongbao and Luo, Xiao and Yang, Tao},
  journal={Annales Fennici Mathematici},
  volume={47},
  pages={895--925},
  year={2022}
}

@article{jeanjean1997existence,
  title={Existence of solutions with prescribed norm for semilinear elliptic equations},
  author={Jeanjean, Louis},
  journal={Nonlinear Analysis: Theory, Methods \& Applications},
  volume={28},
  number={10},
  pages={1633--1659},
  year={1997},
  publisher={Pergamon}
}

@article{soave2020normalized1,
  title={Normalized ground states for the {NLS} equation with combined nonlinearities},
  author={Soave, Nicola},
  journal={Journal of Differential Equations},
  volume={269},
  number={9},
  pages={6941--6987},
  year={2020},
  publisher={Elsevier}
}

@article{soave2020normalized2,
  title={Normalized ground states for the {NLS} equation with combined nonlinearities: the {S}obolev critical case},
  author={Soave, Nicola},
  journal={Journal of Functional Analysis},
  volume={279},
  number={6},
  pages={108610},
  year={2020},
  publisher={Elsevier}
}

@book{lieb2001analysis,
  title={Analysis},
  author={Lieb, Elliott H and Loss, Michael},
  volume={14},
  year={2001},
  publisher={American Mathematical Soc.}
}

@article{brezis1983relation,
  title={A relation between pointwise convergence of functions and convergence of functionals},
  author={Br{\'e}zis, Haim and Lieb, Elliott},
  journal={Proceedings of the American Mathematical Society},
  volume={88},
  number={3},
  pages={486--490},
  year={1983}
}

@article{he2022normalized,
  title={Normalized ground states for the critical fractional {C}hoquard equation with a local perturbation},
  author={He, Xiaoming and R{\u{a}}dulescu, Vicen{\c{t}}iu D and Zou, Wenming},
  journal={The Journal of Geometric Analysis},
  volume={32},
  number={10},
  pages={252},
  year={2022},
  publisher={Springer}
}

@article{zhen2022normalized,
  title={Normalized ground states for the critical fractional {NLS} equation with a perturbation},
  author={Zhen, Maoding and Zhang, Binlin},
  journal={Revista Matem{\'a}tica Complutense},
  pages={1--44},
  year={2022},
  publisher={Springer}
}

@article{wei2022normalized,
  title={Normalized solutions for {S}chr{\"o}dinger equations with critical {S}obolev exponent and mixed nonlinearities},
  author={Wei, Juncheng and Wu, Yuanze},
  journal={Journal of Functional Analysis},
  volume={283},
  number={6},
  pages={109574},
  year={2022},
  publisher={Elsevier}
}

@article{jeanjean2020mass,
  title={A mass supercritical problem revisited},
  author={Jeanjean, Louis and Lu, Sheng-Sen},
  journal={Calculus of Variations and Partial Differential Equations},
  volume={59},
  number={5},
  pages={174},
  year={2020},
  publisher={Springer}
}

@article{bartsch2023existence,
  title={Existence and asymptotic behavior of normalized ground states for {S}obolev critical {S}chr{\"o}dinger systems},
  author={Bartsch, Thomas and Li, Houwang and Zou, Wenming},
  journal={Calculus of Variations and Partial Differential Equations},
  volume={62},
  number={1},
  pages={9},
  year={2023},
  publisher={Springer}
}

@article{di2012hitchhikers,
  title={Hitchhiker’s guide to the fractional {S}obolev spaces},
  author={Di Nezza, Eleonora and Palatucci, Giampiero and Valdinoci, Enrico},
  journal={Bulletin des sciences math{\'e}matiques},
  volume={136},
  number={5},
  pages={521--573},
  year={2012},
  publisher={Elsevier}
}

@article{feng2018stability,
title={Stability of standing waves for the fractional {S}chr{\"o}dinger--{C}hoquard equation},
  author={Feng, Binhua and Zhang, Honghong},
  journal={Computers \& Mathematics with Applications},
  volume={75},
  number={7},
  pages={2499--2507},
  year={2018},
  publisher={Elsevier}
}

@article{bartsch2019multiple,
  title={Multiple normalized solutions for a competing system of {S}chr{\"o}dinger equations},
  author={Bartsch, Thomas and Soave, Nicola},
  journal={Calculus of Variations and Partial Differential Equations},
  volume={58},
  pages={1--24},
  year={2019},
  publisher={Springer}
}

@article{moroz2013groundstates,
  title={Groundstates of nonlinear {C}hoquard equations: existence, qualitative properties and decay asymptotics},
  author={Moroz, Vitaly and Van Schaftingen, Jean},
  journal={Journal of Functional Analysis},
  volume={265},
  number={2},
  pages={153--184},
  year={2013},
  publisher={Elsevier}
}

@article{ghimenti2016nodal,
  title={Nodal solutions for the {C}hoquard equation},
  author={Ghimenti, Marco and Van Schaftingen, Jean},
  journal={Journal of Functional Analysis},
  volume={271},
  number={1},
  pages={107--135},
  year={2016},
  publisher={Elsevier}
}

@article{hu2023normalized,
  title={Normalized solutions to nonlocal {S}chr{\"o}dinger systems with {$L^2$}--subcritical and supercritical nonlinearities},
  author={Hu, Jiaqing and Mao, Anmin},
  journal={Journal of Fixed Point Theory and Applications},
  volume={25},
  number={3},
  pages={77},
  year={2023},
  publisher={Springer}
}

@article{park2011fractional,
  title={Fractional {P}olya-{S}zeg{\"o} inequality},
  author={Park, Young Ja},
  journal={Journal of the Chungcheong Mathematical Society},
  volume={24},
  number={2},
  pages={267--271},
  year={2011},
  publisher={The Chungcheong Mathematical Society}
}

@article{servadei2013variational,
  title={Variational methods for non-local operators of elliptic type},
  author={Servadei, Raffaella and Valdinoci, Enrico},
  journal={Discrete Contin. Dyn. Syst},
  volume={33},
  number={5},
  pages={2105--2137},
  year={2013}
}

@article{brezis1983positive,
  title={Positive solutions of nonlinear elliptic equations involving critical {S}obolev exponents},
  author={Br{\'e}zis, Haim and Nirenberg, Louis},
  journal={Communications on pure and applied mathematics},
  volume={36},
  number={4},
  pages={437--477},
  year={1983},
  publisher={Wiley Subscription Services, Inc., A Wiley Company New York}
}

@article{bartsch2016normalized,
  title={Normalized solutions for a system of coupled cubic {S}chr{\"o}dinger equations on $\mathbb{R}^3$},
  author={Bartsch, Thomas and Jeanjean, Louis and Soave, Nicola},
  journal={Journal de Math{\'e}matiques Pures et Appliqu{\'e}es},
  volume={106},
  number={4},
  pages={583--614},
  year={2016},
  publisher={Elsevier}
}

@article{bartsch2021normalized,
  title={Normalized solutions for a coupled {S}chr{\"o}dinger system},
  author={Bartsch, Thomas and Zhong, Xuexiu and Zou, Wenming},
  journal={Mathematische Annalen},
  volume={380},
  pages={1713--1740},
  year={2021},
  publisher={Springer}
}

@article{gou2018multiple,
  title={Multiple positive normalized solutions for nonlinear {S}chr{\"o}dinger systems},
  author={Gou, Tianxiang and Jeanjean, Louis},
  journal={Nonlinearity},
  volume={31},
  number={5},
  pages={2319},
  year={2018},
  publisher={IOP Publishing}
}

@article{zhang2024normalized,
  title={Normalized ground states for {S}chr{\"o}dinger system with a coupled critical nonlinearity},
  author={Zhang, Penghui and Han, Zhiqing},
  journal={Applied Mathematics Letters},
  volume={150},
  pages={108947},
  year={2024},
  publisher={Elsevier}
}

@article{li2024existence,
  title={Existence of normalized solutions for a fractional Kirchhoff system},
  author={Li, Menghui and He, Jinchun and Xu, Haoyuan and Yang, Meihua},
  journal={Complex Variables and Elliptic Equations},
  pages={1--32},
  year={2024},
  publisher={Taylor \& Francis}
}

@article{zhang2023normalized,
  title={Normalized solutions for critical {C}hoquard systems},
  author={Zhang, Hui and Zhang, Jianjun and Zhong, Xuexiu},
doi = {https://doi.org/10.48550/arXiv.2307.01483},
journal={arXiv preprint arXiv:2307.01483},
  year={2023}
}

@article{chennormalized,
  title={Normalized solutions for fractional {S}chr{\"o}dinger-{C}hoquard systems with Sobolev critical coupled nonlinearity},
  author={Chen, Zilin and Yang, Yang},
  journal={Electronic Journal of Differential Equations},
  volume={2025},
  number={1},
  pages={49--22},
  year={2025}
}

@article{wang2021existence,
  title={Existence of normalized solutions for the coupled {H}artree--{F}ock type system},
  author={Wang, Jun},
  journal={Mathematische Nachrichten},
  volume={294},
  number={10},
  pages={1987--2020},
  year={2021},
  publisher={Wiley Online Library}
}

@article{zhang2025normalized,
  title={Normalized solution for a kind of coupled {K}irchhoff systems},
  author={Zhang, Shiyong and Zhang, Qiongfen},
  journal={Electronic Research Archiv},
  volume={2},
  pages={600--612},
  year={2025}
}

@article{liu2023mass,
  title={A mass supercritical and critical {S}obolev fractional {S}chr{\"o}dinger system},
  author={Liu, Mei-Qi and Li, Quanqing},
  journal={Mathematical Methods in the Applied Sciences},
  volume={46},
  number={3},
  pages={3356--3370},
  year={2023},
  publisher={Wiley Online Library}
}

@article{li2020normalized,
  title={Normalized solutions for a coupled fractional {S}chr{\"o}dinger system in low dimensions},
  author={Li, Meng and He, Jinchun and Xu, Haoyuan and Yang, Meihua},
  journal={Boundary Value Problems},
  volume={2020},
  number={1},
  pages={166},
  year={2020},
  publisher={Springer}
}

@article{giacomoni2018doubly,
  title={Doubly nonlocal system with Hardy--Littlewood--Sobolev critical nonlinearity},
  author={Giacomoni, Jacques and Mukherjee, Tuhina and Sreenadh, K},
  journal={Journal of Mathematical analysis and applications},
  volume={467},
  number={1},
  pages={638--672},
  year={2018},
  publisher={Elsevier}
}

@article{mukherjee2016fractional,
  title={Fractional Choquard equation with critical nonlinearities},
  author={Mukherjee, Tuhina and Sreenadh, K},
  journal={Nonlinear Differential Equations and Applications NoDEA},
  volume={24},
  pages={1--34},
  year={2017},
  publisher={Springer}
}

@article{MENG2025113845,
title = {Normalized solutions to Schrödinger systems with critical nonlinearities},
journal = {Nonlinear Analysis},
volume = {260},
pages = {113845},
year = {2025},
issn = {0362-546X},
author = {Yuxi Meng and Xiaoming He and Patrick Winkert}
}

@article{yang2020normalized,
  title={Normalized solutions for the fractional {S}chr{\"o}dinger equation with a focusing nonlocal {$L^2$}-critical or {$L^2$}- supercritical perturbation},
  author={Yang, Tao},
  journal={Journal of Mathematical Physics},
  volume={61},
  number={5},
  year={2020},
  publisher={AIP Publishing}
}

@article{bhattarai2017fractional,
  title={On fractional {S}chr{\"o}dinger systems of {C}hoquard type},
  author={Bhattarai, Santosh},
  journal={Journal of Differential Equations},
  volume={263},
  number={6},
  pages={3197--3229},
  year={2017},
  publisher={Elsevier}
}

@article{feng2019existence,
  title={Existence of stable standing waves for the fractional {S}chr{\"o}dinger equations with combined power-type and {C}hoquard-type nonlinearities},
  author={Feng, Binhua and Chen, Ruipeng and Ren, Jiajia},
  journal={Journal of Mathematical Physics},
  volume={60},
  number={5},
  year={2019},
  publisher={AIP Publishing}
}

@article{frantzeskakis2010dark,
  title={Dark solitons in atomic {B}ose--{E}instein condensates: from theory to experiments},
  author={Frantzeskakis, DJ},
  journal={Journal of Physics A: Mathematical and Theoretical},
  volume={43},
  number={21},
  pages={213001},
  year={2010},
  publisher={IOP Publishing}
}

@inproceedings{cabre2014nonlinear,
  title={Nonlinear equations for fractional {L}aplacians, {I}: {R}egularity, maximum principles, and {H}amiltonian estimates},
  author={Cabr{\'e}, Xavier and Sire, Yannick},
  booktitle={Annales de l'Institut Henri Poincar{\'e} C, Analyse non lin{\'e}aire},
  volume={31},
  number={1},
  pages={23--53},
 year={2014},
  organization={Elsevier}
}

@article{silvestre2007regularity,
  title={Regularity of the obstacle problem for a fractional power of the {L}aplace operator},
  author={Silvestre, Luis},
  journal={Communications on Pure and Applied Mathematics: A Journal Issued by the Courant Institute of Mathematical Sciences},
  volume={60},
  number={1},
  pages={67--112},
  year={2007},
  publisher={Wiley Online Library}
}

@article{guo2024normalized,
  title={Normalized solutions to fractional mass supercritical Choquard systems},
  author={Guo, Zhenyu and Jin, Wenyan},
  journal={The Journal of Geometric Analysis},
  volume={34},
  number={4},
  pages={104},
  year={2024},
  publisher={Springer}
}

@article{ambrosio2020multiplicity,
  title={Multiplicity of solutions for fractional Schr{\"o}dinger systems in $\mathbb{R}^N$},
  author={Ambrosio, Vincenzo},
  journal={Complex Variables and Elliptic Equations},
  volume={65},
  number={5},
  pages={856--885},
  year={2020},
  publisher={Taylor \& Francis}
}

@article{bagnato2015bose,
  title     = {{B}ose--{E}instein condensation: {T}wenty years after},
  author    = {Bagnato, Vanderlei Salvador and Frantzeskakis, D.\ J. and Kevrekidis, P.\ G. and Malomed, B.\ A. and Mihalache, D.},
  journal   = {Romanian Reports in Physics},
  volume    = {67},
  number    = {1},
  pages     = {5--50},
  year      = {2015},
  address   = {Bucharest},
}

@article{fiscella2018p,
  title={p-fractional Hardy--Schr{\"o}dinger--Kirchhoff systems with critical nonlinearities},
  author={Fiscella, Alessio and Pucci, Patrizia and Zhang, Binlin},
  journal={Advances in Nonlinear Analysis},
  volume={8},
  number={1},
  pages={1111--1131},
  year={2018},
  publisher={De Gruyter}
}

@article{guo2024normalized2,
  title={Normalized solutions of linear and nonlinear coupled {C}hoquard systems with potentials},
  author={Guo, Zhenyu and Jin, Wenyan},
  journal={Annals of Functional Analysis},
  volume={15},
  number={2},
  pages={43},
  year={2024},
  publisher={Springer}
}

@article{cao2017existence,
  title={The existence of solutions with prescribed L2-norm for Kirchhoff type system},
  author={Cao, Xiaofei and Xu, Junxiang and Wang, Jun},
  journal={Journal of Mathematical Physics},
  volume={58},
  number={4},
  year={2017},
  publisher={AIP Publishing}
}

@article{yang2021normalized,
  title={Normalized ground state solutions for Kirchhoff type systems},
  author={Yang, Zuo},
  journal={Journal of Mathematical Physics},
  volume={62},
  number={3},
  year={2021},
  publisher={AIP Publishing}
}

@article{KongChen2022,
  author    = {Kong, Lingzheng and Chen, Haibo},
  title     = {Normalized solutions for nonlinear fractional Kirchhoff type systems},
  journal   = {Topological Methods in Nonlinear Analysis},
  volume    = {60},
  number    = {1},
  pages     = {153--183},
  year      = {2022},
}

\end{document}